\documentclass[10pt]{amsart}
\usepackage{amsmath}
\usepackage{amssymb}
\usepackage{enumerate}
\usepackage{amsbsy}
\usepackage{amsfonts}
\usepackage{color}

\newtheorem{thm}{Theorem}[section]
\newtheorem{lem}[thm]{Lemma}
\newtheorem{prop}[thm]{Propsition}
\newtheorem{cor}[thm]{Corollary}
\newtheorem{defn}[thm]{Definition}
\newtheorem{rem}[thm]{Remark}

\numberwithin{equation}{section}

\newcommand{\h}{\dot{H}^1}
\newcommand{\hr}{\dot{H}_{rad}^1}
\newcommand{\rt}{\mathbb{R}^{\mathrm{3}}}
\newcommand{\s}{S(I)}
\newcommand{\w}{W_1(I)}
\newcommand{\ww}{W_2(I)}
\newcommand{\wi}{W_i(I)}
\newcommand{\al}{\alpha}
\newcommand{\de}{\delta}
\newcommand{\bd}{\bar{\delta}}
\newcommand{\lam}{\lambda}
\newcommand{\les}{\lesssim}
\newcommand{\ljn}{\lam_{j,n}}
\newcommand{\tjn}{t_{j,n}}
\newcommand{\sca}{S_{1,n}}
\newcommand{\gs}{{g_s^\frac1{p_0}}}
\newcommand{\eg}{E_g}
\newcommand{\egc}{E_{g,c}}
\newcommand{\hn}{H_{n,\varepsilon_0}}
\newcommand{\rne}{R_{n,\varepsilon_0}}
\newcommand{\ujn}{\tilde{U}_{j,n}}
\newcommand{\vp}{\varphi}
\newcommand{\sst}{S(I^*)}

\begin{document}

	\title[GWP on focusing energy-critical INLS]{On the global well-posedness of focusing energy-critical inhomogeneous NLS}

	\author[Y. Cho]{Yonggeun Cho}
	\address{Department of Mathematics, and Institute of Pure and Applied Mathematics, Chonbuk National University, Jeonju 54896, Republic of Korea}
	
    \author[S. Hong]{Seokchang Hong}
    \address{Department of Mathematical Sciences, Seoul National University, Seoul 08826, Republic of Korea}
    \email{seokchang11@snu.ac.kr}

	\author[K. Lee]{Kiyeon Lee}
	\address{Department of Mathematics, Chonbuk National University, Jeonju 54896, Republic of Korea}
	\email{leeky@jbnu.ac.kr}

	
	\begin{abstract}
		We consider the focusing energy-critical inhomogeneous nonlinear Schr\"odinger equation:
		$$	iu_t + \Delta u + g|u|^{p-1}u = 0, u(0)= \varphi \in \h,$$
where $0 \le g_i \le |x|^bg \le g_s$, $0 < b < \frac43$, and $p = 5-2b$. On the road map of Kenig-Merle \cite{km} we show the global well-posedness and scattering of radial solutions under energy condition $$E_g(\varphi) < E_g(Q_b),\;\;\mbox{and}\;\; \gs\|\varphi\|_{\h}^2 < \|Q_b\|_{\h}^2,$$ where $Q_b$ is the solution of $\Delta Q_b + |x|^{-b}Q_b^p = 0$, together with scaling condition $|x||\nabla g(x)| \lesssim |x|^{-b}$, variational condition $g_s^\frac2{p-1}(\frac{p+1}2-g_i) \le \frac{p-1}2$, and rigidity condition $-bg(x) \le x\cdot \nabla g(x)$. We also provide sharp finite time blowup results for non-radial and radial solutions. For this we utilize the localized virial identity.
	\end{abstract}

		\maketitle

\section{Introduction}
In this paper, we consider the following Cauchy problem for an inhomogeneous nonlinear Schr\"odinger equation in 3 dimensions:
\begin{eqnarray}\label{maineq}
	\left\{\begin{array}{l}
	iu_t + \Delta u + g|u|^{p-1}u = 0 \;\; \mathrm{in}\;\; \mathbb{R}^{1+3},\\
	u(0) = \varphi \in \h(\rt),
	\end{array} \right.
\end{eqnarray}
where $g \in C^1(\rt\backslash \{0\})$ is the coefficient representing interaction among particles and $p=5-2b$ for $0 < b < 2$.
Here $\h$ denotes the homogenous Sobolev space defined by
$$
\h = \{f \in L_x^6 :  \|f\|_{\h}:=\|\nabla f\|_{L_x^2}<+\infty\}.
$$

The equation \eqref{maineq} with $p = 3$ can be a  model of  dilute BEC when the two-body interactions of the condensate are considered. For this see \cite{bpvt, ts}. Also it has been considered to study the laser guiding in an axially nonuniform plasma channel. See \cite{gill, ss, ts}.

	The energy $\eg$ of the solution to \eqref{maineq} can be defined by
\begin{eqnarray}\label{energy}
E_g(u(t))&:=& \frac12 \|\nabla u(t)\|_{L_x^2}^{2} - \frac1{p+1}\int g|u(t)|^{p+1} dx.
\end{eqnarray}
If the solution is sufficiently smooth, then its energy is expected to be conserved, that is, $\eg(u(t))=\eg(\vp)$ for all $t$ in the existence time interval. This will be treated again briefly in Remark \ref{m-e-cons} below.

As the scaling invariance case $g = |x|^{-b}$, a scaling invariance structure can be set up in $\h$ under a condition on $g$. To be more precise, we assume that
\begin{align}\label{scaling}
0 \le g_i \le |x|^bg(x) \le g_s \;\;  \mbox{and}\;\; |x||\nabla g(x)| \lesssim |x|^{-b} \;\;\mbox{for any} \;\;x \neq 0,
\end{align}
 where $g_i = \inf |x|^bg(x)$ and $ g_s = \sup |x|^bg(x)$. Then the $\h$-scaled function $u_\lambda$ for $\lam > 0$ defined by $u_\lam(t,x) = \lam^\frac12u(\lam^2t,\lam x)$ is also the solution to the equation \eqref{maineq} with the coefficient $g_\lam(x) := \lambda^b g(\lambda x)$ satisfying \eqref{scaling}. Thus we may say that \eqref{maineq} is {\it essentially} energy-critical.

We say that \eqref{maineq} is locally well-posed if there exists a maximal existence time interval $I^*$ such that there exists a unique solution $u \in C(I^*; \h)$ and $u$ depends continuously on the initial data. The local well-posedness(LWP) can be usually shown by a contraction argument based on the Strichartz estimate \cite{cawe}. The problem \eqref{maineq} is also said to be globally well-posed if $I^* = \mathbb R$ and the global solution $u$ is said to scatter in $\h$ if there exists linear solutions $u_\pm$ such that $u \to u_\pm$ in $\h$ as $t \to \pm\infty$. In this paper the solution is said to blow up if $\int_{I^*}\int_{\mathbb R^3}|u(t, x)|^{10}\,dxdt = +\infty$ since the $L_{t, x}^{10}$ norm controls our whole contraction argument. We also use the terminology of finite time blowup when $I^*$ is bounded.

Many authors have studied the global behavior for the 3D inhomogeneous NLS for $0 < b < 2$ and $1 < p < 5 - 2b$. For instance we refer the readers to \cite{fagu, dinh, cam} and references therein. Since $p < 5-2b$, the problem has an energy-subcritical nature. They utilized Holmer-Roudenko's \cite{horou}, localized virial, and Dodoson-Murphy's \cite{domu} arguments, respectively, in the mass-energy intercritical view. Up to now there has not been known about the GWP and scattering for energy-critical equations. In this paper we treat these problems under radial symmetry. %The upper bound $\frac43$ stems from to control the nonlinearity by $L_{t, x}^{10}$ in the proof LWP.
%In appendix we consider GWP for other critical case $|x|^{-b}|u|^{p-1}u$ with $p = 5-2b$ for $0 \le b < \frac43$. Since the stream of proof is almost the same as the case $b = 1$, we just brief on GWP and scattering.

Inspired by the result \cite{chkl2} in which the case $b = 0$ is considered, we build up a global theory in $\h$  for $b > 0$ through the concentration-compactness argument of Kenig-Merle \cite{km}. Our strategy is threefold: (1) Variational estimates by adopting the ground state $Q_b$ (2) Existence and compactness of minimal energy blowup solutions(MEBS) (3) Rigidity to remove MEBS based on the localized virial argument. Here the ground state $Q_b$ means a positive radial solution to the elliptic problem
\begin{align}\label{ell}
\Delta Q_b + |x|^{-b}Q_b^p = 0.
\end{align}
Let $p_0 = \frac{p-1}2\; (= 2 - b)$. Then $Q_b$ is given by the function $Q_b(x)= \left(1+ \frac{|x|^{p_0}}{p_0+1}\right)^{-\frac1{p_0}}$ and is unique up to scaling. For this see Remark 2.1 of \cite{ya}.

For the variational estimates we confine the lower and upper bounds of $|x|^{b}g$ as follows:
\begin{align}\label{var-con}
g_0 := \gs(p_0+1-g_i)  \le p_0.
\end{align}
For the rigidity part we need the {\it radial symmetry} and additional rigidity condition for $g$ such that
\begin{align}\label{rig-con}
-bg(x) \le x\cdot \nabla g(x) \;\;\mbox{for all}\;\;x \neq 0.
\end{align}
These conditions together with \eqref{e-cond} below give us a sharpness between GWP and blowup. The rigidity condition \eqref{rig-con} is equivalent to $(s^bg(s))' \ge 0$ and gets rid of the error term occurring when we deal with the lower bound for the second derivative of localized virial quantity $z_r(t) = \int b_r(x)|u(t)|^2\,dx$ with a smooth function $b_r$ for $r > 0$.

In place of \eqref{ell} one may consider a ground state, the positive radial solution to critical stationary problem of \eqref{maineq}:
\begin{align}\label{g-ell}
\Delta Q + gQ^p = 0.
\end{align}
However, when $g$ satisfies \eqref{scaling} and \eqref{rig-con}, we can show that \eqref{g-ell} has no positive radial solution in $\mathbb R^3$. For this see Proposition \ref{noex}, Remark \ref{b>1-rem}, and Remark \ref{noex-rem} in Appendix. By this reason we take into account the variational estimates based on the ground state $Q_b$ together with \eqref{var-con}. However it does not mean that \eqref{var-con} is optimal. We hope \eqref{var-con} to be extended to the case $g_0 > p_0$.

We may take $g$ satisfying \eqref{scaling}, \eqref{var-con}, and \eqref{rig-con} as follows:
$$
g(s) = \frac 1{s^b}\frac{a(s+d)}{s+c} \;\;\mbox{for}\;\; a > 0,\; 0 \le d \le c, \;c > 0,\; a^\frac1{p_0}\left(p_0+1-\frac {ad}c\right) \le p_0\;\;(\mbox{e.g.} \;\;a = 1, d = 0, c = 1),
$$
$$
g(s) = \frac {h(s)}{s^b}, h(s) = \left\{\begin{array}{ll} a\;(0 \le a < p_0+1), &0 \le s < 1,\\ \mbox{smooth and increasing},\;\; &1 < s < 2,\\\left(\frac{p_0}{p_0+1-a}\right)^{p_0}, &s \ge 2.\end{array}\right.
$$

Now we are ready to state our main result.
%======================================================================================================================================================================
\begin{thm}\label{mainresult}
	Let $0 <  b < \frac43$. Let $g$ be a radial function satisfying \eqref{scaling}, \eqref{var-con}, and \eqref{rig-con}.		
	Suppose that $\varphi \in \hr := \{f \in \h : f\;\;\mbox{is radial}\;\}$,
	\begin{align}\label{e-cond}
	E_g(\varphi) < E_g(Q_b),\;\;\mbox{and}\;\; \gs\|\varphi\|_{\h}^2 < \|Q_b\|_{\h}^2.
	\end{align}
	Then \eqref{maineq} is globally well-posed in $\hr$ and the solution $u$ scatters in $\hr$.
\end{thm}

The upper bound $\frac43$ of $b$ is required to control, by $L_{t,x}^{10}$ norm, the nonlinear terms appearing in LWP, which cannot be circumvented in our argument for the present. The condition \eqref{e-cond} implies the energy trapping and coercivity of energy, that is, $E_g(u) \sim \|u\|_{\h}^2 \sim \|\varphi\|_{\h}^2$. It plays a crucial role in the rigidity part. In order to show the existence and nonexistence of MEBS we develop a profile decomposition for radial data and compactness of MEBS flow under \eqref{e-cond} and \eqref{rig-con}.

%It would be  resolve the rigidity part for the case when $g$ satisfies that $\gs(3-b-g_i) > 2-b$.

%%%%%%%%%%%%%%%%%%%%%%%%%%%%%%%%%%%%%%%%%%%%%%%%%%%%%%%%%%%%%%%%%%%%%%%%%%%%%%%%%%%%%%%%%%%%%%%%%%%%%%%%%%%%%%%%%%%%%%%%%%%%%%%%%%%%%%%%%%%%%%%%%%%%%%%%%%%%%%%%%%%%%%%%%%%%%%%%%%%%%%%%%%%%%%%%%%%%%%%%%%%%%%%%%%%%%%%%%%%%%%%%%%%%%%%%%%%%%%%%%%%%%%%%%%%%%%%%%%%%%%%%%%%%%%%%%%%%%%%%%%
On the other hand to obtain a sharp blowup result we need to control the error term for the upper bound of the second derivative of localized virial quantity. To do so we assume that
\begin{align}\label{virial-con}
x\cdot \nabla g(x) \le (p+1)(k_g-\rho)g(x) \;\;\mbox{for all}\;\;x \neq 0,
\end{align}
where $k_g=\frac{p_0-g_0}{p_0+1-g_0}$ and for some $\rho \ge 0$.
Then we get the following.
\begin{thm}\label{blowup-thm}
	Let $0 < b < \frac43$. Let $g$ be a nonnegative and bounded function satisfying \eqref{scaling}, \eqref{var-con}, and \eqref{virial-con}.		
	\begin{enumerate}
		\item[$(1)$] Suppose that $\varphi \in \h$, $|x|\varphi \in L^2$,
		\begin{align}\label{b-cond}
		E_g(\varphi) < E_g(Q_b),\;\;\mbox{and}\;\; \gs\|\varphi\|_{\h}^2 \ge \|Q_b\|_{\h}^2.
		\end{align}
		Then the solution $u$ to \eqref{maineq} blows up in finite time.
		\item[$(2)$] Suppose that $g$ is radial, $\rho > 0$, and $\varphi \in \hr$ satisfies \eqref{b-cond}.
		Then the radial solution $u$ to \eqref{maineq} blows up in finite time.
	\end{enumerate}
\end{thm}
Note that the radial symmetry is not necessary for $(1)$. The condition \eqref{b-cond} leads us to the inequality $\int (|\nabla u|^2 - (1 -\eta) g|u|^{p+1})\,dx < -C$ for some $0 \le \eta < 1$ and $C > 0$ and hence to the finite time blowup through the localized virial identity \eqref{lvirial} below. This argument also appears in some literatures (see \cite{ken, bole, dinh}). In $(2)$ the moment condition $|x|\varphi \in L_x^2$ has been replaced with the radial symmetry and $L_x^2$ condition. This is due to the space-decay estimate of Strauss \cite{stra}. The condition $\rho > 0$ in $(2)$ is required to handle error terms appearing in localized virial argument.\\

%%%%%%%%%%%%%%%%%%%%%%%%%%%%%%%%%%%%%%%%%%%%%%%%%%%%%%%%%%%%%%%%%%%%%%%%%%%%%%%%%%%%%%%%%%%%%%%%%%%%%%%%%%%%%%%%%%%%%%%%%%%%%%%%%%%%%%%%%%%%%%%%%%%%%%%%%%%%%%%%%%%%%%%%%%%%%%%%%%%%%%%%%%%%%%%%%%%%%%%%%%%%%%%%%%%%%%%%%%%%%%%%%%%%%%%%%%%%%%%%%%%%%%%%%%%%%%%%%%%%%%%%%%%%%%%%%%%%%%%%%%%%%%%%%%%%%%%%%%%%%%%%%%%%%%%%%%%%%%%%%%%%%%%%%%%%%%%%%%%%%%%%%%%%%%%%%%%%%%%%%%%%%%%%%%%%%%%%%%%%%%%%%%%%%%%%%%%%%%%%%%%%%%%%%%%%%%%%%%%%%%%%%%%%%%%%%%%%%%%%%%%%%%%%%%%%%%%%%%%%%%%%%%%%%%%%%%%%%%%%%%%%%%%%%%%%%%%%%%%%%%%%%%%%%%%%%%%%%%%%%%%%%%%%%%%%%%%%%%%%%%%%%%%%%%%%%%%%%%%%%%%%%%%%%%%%%%%%%%%%%%%%%%%%%%%%%%%%%%%%%%%%%%%%%%%%%%%%%%%%%%%%%%%%%%%%
\noindent\textbf{Notations.}\\
%%%%%%%%%%%%%%%%%%%%%%%%%%%%%%%%%%%%%%%%%%%%%%%%%%%%%%%%%%%%%%%%%%%%%%%%%%%%%%%%%%%%%%%%%%%%%%%%%%%%%%%%%%%%%%%%%%%%%%%%
\noindent$\bullet$ Mixed-normed spaces: For a Banach space $X$ and an interval $I$, $u \in L_I^q X$ iff $u(t) \in X$ for a.e. $t \in I$ and $\|u\|_{L_I^qX} := \|\|u(t)\|_X\|_{L_I^q} < \infty$. Especially, we denote  $L_I^qL_x^r = L_t^q(I; L_x^r(\rt))$, $L_{I, x}^q = L_I^qL_x^q$, $L_t^qL_x^r = L_{\mathbb R}^qL_x^r$,
$$
S(I)=L_I^{10}L_x^{10},\;\; W_1(I)=L_I^{10}L_x^{\frac{30}{13}}, \;\;\mbox{and}\;\;  W_2(I)=L_I^{\frac{10(b+1)}{3b+1}}L_x^{\frac{30(b+1)}{9b+13}}.
$$
%%%%%%%%%%%%%%%%%%%%%%%%%%%%%%%%%%%%%%%%%%%%%%%%%%%%%%%%%%%%%%%%%%%%%%%%%%%%%%%%%%%%%%%%%%%%%%%%%%%%%%%%%%%%%%%%%%%%%%%%%
\noindent$\bullet$ As usual different positive constants depending only on $b, g_i, g_s$ are denoted by the same letter $C$, if not specified. $A \lesssim B$ and $A \gtrsim B$ means that $A \le CB$ and
$A \ge C^{-1}B$, respectively for some $C>0$. $A \sim B$ means that $A \lesssim B$ and $A \gtrsim B$.\\

%%%%%%%%%%%%%%%%%%%%%%%%%%%%%%%%%%%%%%%%%%%%%%%%%%%%%%%%%%%%%%%%%%%%%%%%%%%%%%%%%%%%%%%%%%%%%%%%%%%%%%%%%%%%%%%%%%%%%%%%%%%%%%%%%%%%%%%%%%%%%%%%%%%%%%%%%%%%%%%%%%%%%%%%%%%%%%%%%%%%%%%%%%%%%%%%%%%%%%%%%%%%%%%%%%%%%%%%%%%%%%%%%%%%%%%%%%%%%%%%%%%%%%%%%%%%%%%%%%%%%%%%%%%%%%%%%%%%%%%%%%%%%%%%%%%%%%%%%%%%%%%%%%%%%%%%%%%%%%%%%%%%%%%%%%%%%%%%%%%%%%%%%%%%%%%%%%%%%%%%%%%%%%%%%%%%%%%%%%%%%%%%%%%%%%%%%%%%%%%%%%%%%%%%%%%%%%%%%%%%%%%%%%%%%%%%%%%%%%%%%%%%%%%%%%%%%%%%%%%%%%%%%%%%%%%%%%%%%%%%%%%%%%%%%%%%%%%%%%%%%%%%%%%%%%%%%%%%%%%%%%%%%%%%%%%%%%%%%%%%%%%%%%%%%%%%%%%%%%%%%%%%%%%%%%%%%%%%%%%%%%%%%%%%%%%%%%%%%%%%%%%%%%%%%%%%%%%%%%%%%%%%%%%%%%%%%%%%%%%%%%%%%%%%%%%%%%%%%%%%%%%%%%%%%%%%%%%%%%%%%%%%%%%%%%%%%%%%%%%%%%%%%%%%%%%%%%%%%%%%%%%%%%%%%%%%%%%%%%%%%%%%%%%%%%%%%%%%%%%%%%%%%%%%
\section{Locat theory}

We first introduce some preliminaries which will be useful in local and global theories.
By Duhamel's principle the equation \eqref{maineq} is rewritten as the integral equation:
	\begin{eqnarray}\label{inteq}
	u = e^{it\Delta}\varphi + i \int_{0}^{t} e^{i(t-t')\Delta}g|u(t')|^{p-1}u(t')dt'.
	\end{eqnarray}
	Here we define the linear propagator $e^{it\Delta}$ given by the solution to the linear problem $i\partial_tv=-\Delta v$ with initial data $v(0)=f$. It is formally given by
	$$e^{it\Delta}f = \mathcal{F}^{-1}\left(e^{-it|\xi|^2}\mathcal{F}(f)\right)= (2\pi)^{-3}\int_{\rt} e^{i(x\cdot\xi - t|\xi|^2)}\widehat{f}(\xi)d\xi,$$
	where $\widehat{f} = \mathcal F( f)$ denotes the Fourier transform of $f$ and $\mathcal F^{-1}(h)$ the inverse Fourier transform of $h$ such that
	$$\mathcal{F}(f)(\xi) = \int_{\rt} e^{- ix\cdot \xi} f(x)\,dx,\quad \mathcal F^{-1} (h)(x) = (2\pi)^{-3}\int_{\rt} e^{ix\cdot \xi} h(\xi)\,d\xi.$$

\begin{lem}[\cite{kt}]\label{str}
	Let $(q,r)$ and $(\widetilde q, \widetilde r)$ be pairs such that $2\le q,r,\widetilde q,\widetilde r \le \infty$ and satisfy the equation $\frac2q+\frac3r = \frac32$.
	Then we have
	\begin{align*}
	\|e^{it{\Delta}} \varphi\|_{L_t^{q} L_x^r} &\lesssim \| \varphi\|_{L_x^2},\\
	\left\|\int_0^t e^{i(t-t')\Delta}F\,dt'\right\|_{L_t^qL_x^r} &\lesssim \|F\|_{L_t^{\widetilde q'}L_x^{\widetilde r'}}.\\
	\end{align*}
\end{lem}
We call such pair {\it admissible} one. The pairs $(10, \frac{30}{13})$ and $(\frac{10(b+1)}{3b+1}, \frac{30(b+1)}{9b+13})$ are admissible.

\begin{lem}[\cite{caze}]\label{hs-ineq}
	Let $f \in \dot W^{1, p}(\mathbb R^3) = \{f \in L^\frac{np}{n-p} : \|f\|_{\dot W^p} := \|\nabla f\|_{L_x^p} < \infty\}$. Then for $1 < p < n$ we have
	$$
	\||x|^{-1}f\|_{L_x^p} \lesssim \|f\|_{ \dot W^{1, p}}.
	$$
\end{lem}

\subsection{Local well-posedness}\label{local-t}

We have only to show the following LWP for \eqref{inteq}.

\begin{prop}\label{lwp}
	Let $\varphi\in\h, 0\in I$ an interval, and $0 < b < \frac43$. Assume that $\|\varphi\|_{\h}\le A$. Then there exists $\delta = \delta(A)$ satisfied following:	
	If $\|e^{it\Delta}\varphi\|_{\s} < \delta$, then there exists a unique solution $u$ of \eqref{maineq} in $I \times \rt$ with $u\in C(I;\h(\rt))$,
	$$
	\|u\|_{\s}\le 2\delta,\;\;\mbox{and} \;\; \|\nabla u\|_{\wi}<\infty\;\;(i=1,2).
	$$
	
	In particular, if $\varphi_k \to \varphi$ in $\h$, then the corresponding solutions $u_k \to u$ in $C(I; \h)$ as $k \to \infty$.
\end{prop}

\begin{proof}[Proof of Proposition \ref{lwp}]

	We use the contraction mapping principle. To this end we fix $r, s > 0$, to be chosen later. Let us define a complete metric space $(B_{r, s}, d)$ and a mapping $\Phi$ as following:
	\begin{align*}
	&B_{r, s} = \{v \in C(I; \h) : \|v\|_{L_I^\infty \h} \le 2A, \;\;\|v\|_{\s}\le r,\;\; \|\nabla v\|_{W_i(I)}\le s \;\;(i=1,2)\},\\
    &d(u, v) = \|v - v'\|_{L_I^\infty \h} + \|v-v'\|_{\s} + \sum_{i = 1,2}\|\nabla(v-v')\|_{W_i(I)},\\
	&\Phi(v) = e^{it\Delta}\varphi + i \int_{0}^{t} e^{i(t-t')\Delta}f(v)dt',\quad f(v) = g|v|^{p-1}v, \quad p = 5-2b.
	\end{align*}
		By the scaling condition \eqref{scaling}, Lemma \ref{str}, and Lemma \ref{hs-ineq} we obtain for each $i=1, 2$ that
	\begin{align}
	\|\nabla\Phi(v)\|_{W_i(I)} &\le C\|\nabla \varphi\|_{L_x^2} + C\|\nabla f(v)\|_{L_I^2L_x^\frac65}\nonumber\\
	&\le C\left(A+\|v\|_{\s}^{p-1-b}\||x|^{-1} v\|_{\ww}^{b+1} + \|v\|_{\s}^{p-1-b}\||x|^{-1}v\|_{\ww}^b\|\nabla v\|_{\ww}\right)\label{L10}\\
	&\le C(A+r^{p-1-b}s^{b+1}).\nonumber
	\end{align}
Here we used the H\"older pairs such that
$$
\frac56 = \frac{p-1-b}{10} + \frac{9b+13}{30},  \;\;\frac12 = \frac{p-1-b}{10} + \frac{3b+1}{10}.
$$
This choice is plausible because $p - 1 - b = 4-3b > 0$.
	
	Choosing $s=2AC$ and $Cr^{p-1-b}s^b \le \min(\frac12, \frac1{2C})$, we have $\|\nabla\Phi(v)\|_{W_i(I)} \le s$ for $i=1,2$. Now if $\delta = \frac r2$ and $Cr^{3-3b}s^{b+1}\le\frac12$, Combining Sobolev embedding $\dot{W}^{1,\frac{30}{13}}\hookrightarrow L^{10}$ and the same argument as above, we deduce that
	$$\|\Phi(v)\|_{\s} \le \|e^{it\Delta}\varphi\|_{\s} + C\|\nabla f(v)\|_{L_I^2L_x^\frac65} \le \delta + Cr^{p-1-b}s^{b+1} \le r = 2\delta.$$
	$$\|\Phi(v)\|_{L_I^\infty \h} \le A + Cr^{p-1-b}s^{b+1} \le 2A$$
Hence $\Phi$ is self-mapping on $B_{r, s}.$
	
	Next we show $\Phi$ is a contraction map.
	\begin{align*}
	d(\Phi(v), \Phi(v')) &\lesssim \||x|^{-b-1}(|v|^{p-1}v-|v'|^{p-1}v')\|_{L_I^2L_x^\frac65} + \||x|^{-b}\nabla(|v|^{p-1}v-|v'|^{p-1}v')\|_{L_I^2L_x^\frac65}\\
	&\le \|(|v|^{p-1} + |v'|^{p-1})|x|^{-b-1}|v - v'|\|_{L_I^2L_x^\frac65} + \||x|^{-b}|v|^{p-1}|\nabla v - \nabla v'|\|_{L_I^2L_x^\frac65}\\
    &\qquad + \||x|^{-b} |\nabla v'|(|v|^{p-2} + |v'|^{p-2})|v-v'|\|_{L_I^2L_x^\frac65}\\
    & =: N_1 + N_2 + N_3.
    \end{align*}
    Similarly to \eqref{L10} $N_1$ and $N_2$ estimate
    \begin{align*}
    N_1 &\lesssim (\|v\|_{\s}^{p-1} + \|v'\|_{\s}^{p-1})\||x|^{-1}(v-v')\|_{\ww}^{b+1},\\
    N_2 &\lesssim  \|v\|_{\s}^{p-1-b}\||x|^{-1}v\|_{\ww}^b\|\nabla(v-v')\|_{\ww}.
    \end{align*}
    $N_3$ is handled differently by the value of $b$. If $1 \le b < \frac43$, then
    $$N_3 \lesssim \|\nabla v'\|_{\ww}(\|v\|_{\s} + \|v'\|_{\s})^{p-1-b}(\||x|^{-1}v\|_{\ww} + \||x|^{-1}v'\|_{\ww})^{b-1}\||x|^{-1}(v-v')\|_{\ww}.$$
    If $0 < b < 1$, then
    $$N_3 \lesssim \|\nabla v'\|_{\ww}(\|v\|_{\s} + \|v'\|_{\s})^{p-2}\|v-v'\|_{\s}^{1-b}\||x|^{-1}(v-v')\|_{\ww}^b.$$
	The above estimates yields that
\begin{align*}
d(\Phi(v), \Phi(v'))\le C (r^{p-1-b}s^{b} + r^{p-2}s) d(v, v').
	\end{align*}
Therefore $\Phi$ is a contraction map provided $C (r^{p-1-b}s^{b} + r^{p-2}s)< 1$.
	
	The continuous dependency on initial data follows immediately from the above contraction argument. This completes the proof of Proposition \ref{lwp}.
\end{proof}

%\begin{rem}
%In view of \eqref{L10} our contraction argument above relies on $\|v\|_{\s}$ to control nonlinear term. It is required that $p-1-b = 4-3b > 0$. This is the reason why we assume $b < \frac43$.
%\end{rem}

\begin{rem}[blowup criterion]\label{blowup-cri}
	Proposition \ref{lwp} implies the existence of maximal existence time interval $I^*$. Moreover, one can immediately deduce the blowup criterion: if $\|u\|_{\sst} < +\infty$, then $I^* = \mathbb R$, and if $I^*$ is bounded, then $\|u\|_{\sst} = +\infty$. We also conclude that if $\|\varphi\|_{\h}$ is sufficiently small, then $I^* = \mathbb R$.
\end{rem}

\begin{rem}[$\h$ scattering]\label{scattering}
	Suppose that $I^* = \mathbb R$ and $\|u\|_{\sst} < +\infty$. Let us set
	$$
	\varphi_\pm := \varphi + i\int_0^{\pm \infty} e^{-it'\Delta}[ g|u|^{p-1}u]\,dt'.
	$$
	Then the solution $u$ scatters to $e^{it\Delta}\varphi_\pm$ in $\h$ by standard duality argument.
\end{rem}

\begin{rem}[mass-energy conservation]\label{m-e-cons}
Let us define the mass by $\|u(t)\|_{L_x^2}^2$ for the solution $u$ to \eqref{maineq}. If we assume that $\varphi \in H^1$, we can readily show the LWP in $H^1$ by the similar way to Proposition \ref{lwp} and also show that $u, \nabla u \in L_{[0,T]}^q L_x^r$ for any admissible pair $(q, r)$ and for any $[0, T] \subset I^*$ by using Duhamel's formula \eqref{inteq}. At this point we can apply Ozawa's argument in \cite{ozawa} directly to show the mass conservation and energy conservation without further regularizing argument. In addition, even though $\varphi \in \h$, the energy conservation follows from the standard density argument ($H^1 \hookrightarrow \h$) and continuous dependency on the initial data.
\end{rem}

%%%%%%%%%%%%%%%%%%%%%%%%%%%%%%%%%%%%%%%%%%%%%%%%%%%%%%%%%%%%%%%%%%%%%%%%%%%%%%%%%%%%%%%%%%%%%%%%%%%%%%%%%%%%%%%%%%%%%%%%%%%%%%%%%%%%%%%%%%%%%%%%%%%%%%%%%%%%%%%%%%%%%%%%%%%%%%%%%%%%%%%%%%%%%%%%%%%%%%%%%%%%%%%%%%%%%%%%%%%%%%%%%%%%%%%%%%%%%%%%%%%%%%%%%%%%%%%%%%%%%%%%%%%%%%%%%%%%%%%%%%%%%%%%%%%%%%%%%%%%%%%%%%%%%%%%%%%%%%%%%%%%%%%%%%%%%%%%%%%%%%%%%%%%%%%%%%%%%%%%%%%%%%%%%%%%%%%%%%%%%%%%%%%%%%%%%%%%%%%%%%%%%%%%%%%%%%%%%%%%%%%%%%%%%%%%%%%%%%%%%%%%%%%%%%%%%%%%%%%%%%%%%%%%%%%%%%%%%%%%%%%%%%%%%%%%%%%%%%%%%%%%%%%%%%%%%%%%%%%%%%%%%%%%%%%%%%%%%%%%%%%%%%%%%%%%%%%%%%%%%%%%%%%%%%%%%%%%%%%%%%%%%%%%%%%%%%%%%%%%%%%%%%%%%%%%%%%%%%%%%%%%%%%%%%%%%%%%%%%%%%%%%%%%%%%%%%%%%%%%%%%%%%%%%%%%%%%%%%%%%%%%%%%%%%%%%%%%%%%%%%%%

\subsection{Long-time perturbation}\label{per-t}
\begin{prop}\label{per}
	Let $g$ be a radial function satisfying \eqref{scaling} with $0 < b < \frac43$. Let $I\subset \mathbb{R}$ be a time interval containing $0$ and $\tilde{u}$ be a radial function defined on $I \times \rt$. Assume that $\tilde{u}$ satisfies following:
	$$
	\|\tilde{u}\|_{L_t^{\infty}\h} \le A \;\;\mbox{and}\;\; \|\tilde{u}\|_{\s}\le M
	$$
	for some constants $M, A > 0$ and
	$$
	i\partial_t\tilde{u} + \Delta\tilde{u} + f(\tilde{u}) = e \;\;\mbox{for}\;\; (t,x)\in I\times \rt,
	$$
	where $f(\tilde u)=g|\tilde u|^{p-1}\tilde u$ and that
	$$
	\|\varphi-\tilde{u}(0)\|_{\h} \le A',\; \|\nabla e\|_{L_I^2L_x^{\frac65}}\le \varepsilon,\;\;\mbox{and}\;\; \|\nabla e^{it\Delta}[\varphi-\tilde{u}(0)]\|_{\wi}\le \varepsilon \;\;(i=1,2).
	$$
		Then there exists $\varepsilon_0 = \varepsilon_0(M,A,A')$ and a unique solution $u \in C(I; \hr)$ with $u(0)=\varphi$ in $I\times \mathbb{R}$, such that for $0 < \varepsilon < \varepsilon_0$ with
	$$
	\|u\|_{\s} \le C(M,A,A')\;\;\mbox{and}\;\; \|u(t)-\tilde{u}(t)\|_{\h} \le A' + C(M,A,A')\varepsilon\;\;\mbox{for all}\;\; t \in I.
	$$

\end{prop}

\begin{proof}
	We may assume, without loss of generality, that $I = [0, a)$ for some $0 < a \le +\infty$. H\"older's and Hardy-Sobolev's inequalities (Lemma \ref{hs-ineq}) yield
	$$
	\|\nabla(g|u|^{p-1}u)\|_{L_I^2L_x^\frac65} \les \|u\|_{\s}^{p-1-b}\|\nabla u\|_{\ww}^{b+1}.
	$$
	Indeed, by the integral equation \eqref{inteq} for $\tilde u$ and Proposition \ref{lwp}, we obtain
	$$
	\|\nabla \tilde{u}\|_{W_{i}(I_k)} \le CA + \rho\|\nabla \tilde{u}\|_{W_i(I_k)}^{b+1}
	$$
	for $\{I_k\}$ satisfying $\bigcup I_k = I$ and $C\|\tilde{u}\|_{S(I_k)}^{4-3b} \le \rho$. Hence, by continuous argument $\|\nabla \tilde{u}\|_{W_i(I_k)} < 2CA$ for sufficiently small $\rho$ and hence one can readily obtain
	$$
	\|\nabla \tilde{u}\|_{\wi} \le \widetilde{M}
	$$
	for some $\widetilde M$ depending on $M, A$.
	
	Let us define $u = \tilde{u} +w$, so that the equation for $w$ is written as
	\begin{eqnarray*}
		\left\{\begin{array}{l}
			iw_t + \Delta w = f(\tilde{u} +w) -f(\tilde{u}) + e, \\
			w(0) = \varphi - \tilde{u}(0).
		\end{array} \right.
	\end{eqnarray*}
	Then for arbitrary $\eta > 0$, there exists $I_j= \left[a_j,a_{j+1}\right)$ such that $\bigcup_{j=1}^JI_j = I$ and $\|\nabla \tilde{u}\|_{W_i(I_j)}\le \eta \;(i=1,2)$. On $I_j$ $w$ satisfies
	$$
	w(t) = e^{i(t-a_j)\Delta}w(a_j) + i \int_{a_j}^{t}e^{i(t-t')\Delta}(f(\tilde{u} +w) -f(\tilde{u}))dt' - i\int_{a_j}^{t}e^{i(t-t')\Delta}e(t')dt'.
	$$
	By Sobolev embedding and Lemma \ref{hs-ineq}, we get
	\begin{align*}
	\sum_{i=1}^{2}\|\nabla w\|_{W_i(I_j)} &\le \left(\sum_{i=1}^{2}\|\nabla e^{i(t-a_j)\Delta}w(a_j)\|_{W_i(I_j)}+2C\varepsilon\right) + C\eta^{b+1}\sum_{i=1}^{2}\|\nabla w\|_{W_i(I_j)}\\
	& \qquad\quad   +C\left(\sum_{i=1}^{2}\|\nabla w\|_{W_i(I_j)}\right)^p.
	\end{align*}
	Thus, if $C\eta^2 \le \frac13$, we have	
	$$
	\sum_{i=1}^{2}\|\nabla w\|_{W_i(I_j)} \le \frac32\gamma_j + C\left(\sum_{i=1}^{2}\|\nabla w\|_{W_i(I_j)}\right)^p,
	$$
	where $\gamma_j = \sum_{i=1}^{2}\|\nabla e^{i(t-a_j)\Delta}w(a_j)\|_{W_i(I_j)}+ 2C\varepsilon$.
	
	From the standard continuity argument, we can find $C_0 > 0$ satisfying that
	$$
	\sum_{i=1}^{2}\|\nabla w\|_{W_i(I_j)} \le 3 \gamma_j\;\;\mbox{and}\;\; C \left(\sum_{i=1}^{2}\|\nabla w\|_{W(I_j)}\right)^p \le 3\gamma_j,
	$$
	provided $\gamma_j \le C_0$. Repeating the above argument for the equation
	\begin{align*}
	e^{i(t-a_{j+1})\Delta}w(a_{j+1}) = e^{i(t-a_j)\Delta}w(a_j) + i \int_{a_j}^{a_{j+1}}e^{i(t-t')\Delta}(f(\tilde{u} +w) -f(\tilde{u}))dt'\\
	- i\int_{a_j}^{a_{j+1}}e^{i(t-t')\Delta}e(t')dt',
	\end{align*}
	we get
	\begin{align*}
	\sum_{i=1}^{2}\|e^{i(t-a_{j+1})\Delta}w(a_{j+1})\|_{W_i(I_{j+1})} &\le \sum_{i=1}^{2}\|\nabla e^{i(t-a_j)\Delta}w(a_j)\|_{W_i(I_{j+1})} + C\varepsilon\\
	&\qquad \qquad+ C\eta^2\sum_{i=1}^{2}\|\nabla w\|_{W_i(I_{j+1})} + C\left(\sum_{i=1}^{2}\|\nabla w\|_{W_i(I_{j+1})}\right)^p.
	\end{align*}
	Taking a sufficiently small $\eta$ to satisfy $\gamma_{j+1} \le 10 \gamma_j$ provided $\gamma_j \le C_0$. This always happens if $C10^{J}\varepsilon_0 < C_0$. With this $\varepsilon_0$ we have that for any $0 < \varepsilon < \varepsilon_0$
	$$
	\|w\|_{\s} + \|\nabla w\|_{\w} + \|\nabla w\|_{\ww} \le 3C\sum_{j = 1}^J\gamma_j \le \frac{C}3(10^{J+1}-1)\varepsilon.
	$$
	Hence by setting $C(M, A, A') = C(10^{J+1}-1)\varepsilon_0/3$ we obtain
	$$
	\|u\|_{\s} \le \|w\|_{\s} + \|\tilde u\|_{\s} \le C(M, A, A').
	$$
	
	Using the Strichartz estimate and Hardy-Sobolev inequality once more, we reach that
	$$
	\|w\|_{L_I^\infty \h} \le A' + C\varepsilon + C\sum_{j = 1}^J\|\nabla(f(\tilde u + w) - f(\tilde u))\|_{L_I^2L_x^{\frac65}} \le A' + C(M, A, A')\varepsilon.
	$$
\end{proof}

%%%%%%%%%%%%%%%%%%%%%%%%%%%%%%%%%%%%%%%%%%%%%%%%%%%%%%%%%%%%%%%%%%%%%%%%%%%%%%%%%%%%%%%%%%%%%%%%%%%%%%%%%%%%%%%%%%%%%%%%%%%%%%%%%%%%%%%%%%%%%%%%%%%%%%%%%%%%%%%%%%%%%%%%%%%%%%%%%%%%%%%%%%%%%%%%%%%%%%%%%%%%%%%%%%%%%%%%%%%%%%%%%%%%%%%%%%%%%%%%%%%%%%%%%%%%%%%%%%%%%%%%%%%%%%%%%%%%%%%%%%%%%%%%%%%%%%%%%%%%%%%%%%%%%%%%%%%%%%%%%%%%%%%%%%%%%%%%%%%%%%%%%%%%%%%%%%%%%%%%%%%%%%%%%%%%%%%%%%%%%%%%%%%%%%%%%%%%%%%%%%%%%%%%%%%%%%%%%%%%%%%%%%%%%%%%%%%%%%%%%%%%%%%%%%%%%%%%%%%%%%%%%%%%%%%%%%%%%%%%%%%%%%%%%%%%%%%%%%%%%%%%%%%%%%%%%%%%%%%%%%%%%%%%%%%%%%%%%%%%%%%%%%%%%%%%%%%%
\section{Variational estimates}\label{sec-vari}
%Let us define the ground state $Q_b$ as following:
%\begin{align}\label{ell}
%Q_b(x)= \left(1+\frac{|x|}2\right)^{-1}
%\end{align}
%which is the positive radial solution to the elliptic equation:
%$$
%\Delta Q_b + |x|^{-1}|Q_b|^2Q_b=0.
%$$

We now provide some variational inequalities showing a sharpness between GWP and blowup. Let $C_*$ be the best constant satisfying $\||x|^{-\frac b{p+1}}u\|_{L_x^{p+1}} \le C_*\|u\|_{\h}$. The existence of minimizer $u_*$ such that $\||x|^{-\frac b{p+1}}u\|_{L_x^{p+1}} = C_*\|u_*\|_{\h}$ is well-known. For instance this see Theorem 4.3 of \cite{lieb}. By the standard variational argument one can show that $u_*$ can be characterized by $u_* = e^{i\theta}\lambda^\frac12 Q_b(\lambda(x))$ for some $\theta \in [-\pi, \pi]$, $\lambda > 0$, and $x_0 \in \mathbb R^3$. From the elliptic equation \eqref{ell} it follows that $\int |\nabla Q_b|^2 = \int |x|^{-b}|Q_b|^{p+1}$. Hence, $\int |\nabla Q_b|^2 = \frac{1}{C_*^{p+1}}$.
%%%%%%%%%%%%%%%%%%%%%%%%%%%%%%%%%%%%%%%%%%%%%%%%%%%%%%%%%%%%%%%%%%%%%%%%%%%%%%%%%%%%%%%	
\begin{lem}[Energy trapping]\label{trapping}
	Let $u$ be a solution of \eqref{maineq} with $\varphi$  such that
	$$
	\gs\|\varphi\|_{\h}^2 < \|Q_b\|_{\h}^2\;\;\mbox{and}\;\; \eg(\varphi)\le(1-\delta_0)\eg(Q_b)
	$$
	for some $\delta_0 >0$. Then there exits $\bar{\delta}=\bar{\delta}(\delta_0)$ such that
	\begin{align*}
	&(i)   \;\gs\|u(t)\|_{\h}^2 \le (1-\bar{\delta})\|Q_b\|_{\h}^2,\\
	&(ii)  \;\int |\nabla u(t)|^2 - g|u(t)|^{p+1} dx \ge \bar{\delta}\int |\nabla u(t)| dx,\\
	&(iii) \;{\rm (Coercivity)}\;\eg(u(t)) \sim \|u(t)\|_{\h}^2 \sim \|\varphi\|_{\h}^2,
	\end{align*}
	for all $t \in I^*$, where $I^*$ is the maximal existence time interval.
\end{lem}
The above lemma follows from the continuity argument, energy conservation, and the following lemma on the initial energy trapping.
%%%%%%%%%%%%%%%%%%%%%%%%%%%%%%%%%%%%%%%%%%%%%%%%%%%%%%%%%%%%%%%%%%%%%%%%%%%%%%%%%%%%%%%%%%%%%%%%%%%%%%%%%%%%%%%%%%%%%%%%%%%%%%%%%%%%%%%%%%%%%%%%%%%%%%%%%%%%%%%%%%%%%%%%%%%%%%%%%%%%%%%%%%%%%%%%%%%%%%%%%%%%%%%%%%%%%%%%%%%%%%%%%%%%%%%%%%%%%%%%%%%%%%%%%%%%%%%%%%%%%%%%%%%%%%%%%%%%%%%%%%%%%%%%%%%%%%%%%%%%%%%%%%%%%%%%%%%%%%%%%%%%%%%%%%%%%%%%%%%%%%%%%%%%%%%%%%%%%%%%%%%%%%%%%%%%%%%%%%%%%%%%%%%%%%%%%%%%%%%%%%%%%%%%%%%%%%%%%%%%%%%%%%%%%%%%%%%%
\begin{lem}\label{vari-lem}
	Let $\eg(\varphi) \le (1-\delta_0)\eg(Q_b), \gs\|\varphi\|_{\h}^2 < \|Q_b\|_{\h}^2$ for some $\delta_0 > 0$. Then there exists $\bar{\delta}=\bar{\delta}(\delta_0)$ such that
	\begin{align*}
	&(i)  \; \gs\|\varphi\|_{\h}^2 \le (1- \bar{\delta})\|Q_b\|_{\h}^2,\\
	&(ii) \; \int |\nabla \varphi|^2 - g|\varphi|^{p+1} dx \ge \bar{\delta}\|\varphi\|_{\h}^2,\\
	&(iii) \; \eg(\varphi) \ge 0.
	\end{align*}
	
\end{lem}
\begin{proof}[Proof of Lemma \ref{vari-lem}]
	Let $f(y)= \frac12y - \frac{C_*^{p+1}}{p+1}y^{\frac{p+1}2}$, $\bar{y} = \gs\|\varphi\|_{\h}^2$, and $ y_0 = \|Q_b\|_{\h}^2$.
	Then we estimate the following:
	\begin{align*}
	f(\bar{y}) &= \frac{\gs}{2}\|\varphi\|_{\h}^2 -\frac{(\gs)^\frac{p+1}2C_*^{p+1}}{p+1}\|\varphi\|_{\h}^{p+1} \le \frac{\gs}{2} \|\varphi\|_{\h}^2 - \frac{(\gs)^\frac{p+1}2}{p+1}\||x|^{-\frac b{p+1}}\varphi\|_{L_x^{p+1}}^{p+1}\\
	&\le \gs \eg(\varphi) \le \gs(2-g_i)(1-\delta_0)f(y_0) \le (1-\delta_0)f(y_0).
	\end{align*}
	Since $0\le \bar{y}<y_0$ and $f$ is strictly increasing on $[0, y_0]$, for some $\bar \delta > 0$ we get
	$$0 \le f(\bar{y}), \quad \bar{y} \le (1-\bar{\delta})y_0.$$
	In particular,
	\begin{align*}
	\int |\nabla\varphi|^2 -g|\varphi|^{p+1} &\ge \int |\nabla\varphi|^2 -g_s C_*^{p+1}\left(\int |\nabla\varphi|^2\right)^{\frac{p+1}2}  \ge  \bar{\delta}\|\varphi\|_{\h}^2.
	\end{align*}
	This completes the proof of Lemma \ref{vari-lem}.
\end{proof}

%%%%%%%%%%%%%%%%%%%%%%%%%%%%%%%%%%%%%%%%%%%%%%%%%%%%%%%%%%%%%%%%%%%%%%%%%%%%%%%%%%%%%%%%%%%%%%%%%%%%%%%%%%%%%%%%%%%%%%%%%%%%%%%%%%%%%%%%%%%%%%%%%%%%%%%%%%%%%%%%%%%%%%%%%%%%%%%%%%%%%%%%%%%%%%%%%%%%%%%%%%%%%%%%%%%%%%%%%%%%%%%%%%%%%%%%%%%%%%%%%%%%%%%%%%%%%%%%%%%%%%%%%%%%%%%%%%%%%%%%%%%%%%%%%%%%%%%%%%%%%%%%%%%%%%%%%%%%%%%%%%%%%%%%%%%%%%%%%%%%%%%%%%%%%%%%%%%%%%%%%%%%%%%%%%%%%%%%%%%%%%%%%%%%%%%%%%%%%%%%%%%%%%%%%%%%%%%%%%%%%%%%%%%%%%%%%%%%%%%%%%%%%%%%%%%%%%%%%%%%%%%%%%%%%%%%%%%%%%%%%%%%%%%%%%%%%%%%%%%%%%%%%%%%%%%%%%%%%%%%%%
%%%%%%%%%%%%%%%%%%%%%%%%%%%%%%%%%%%%%%%%%%%%%%%%%%%%%%%%%%%%%%%%%%%%%%%%%%%%%%%%%%%%%%%%%%%%%%%%%%%%%%%%%%%%%%%%%%%%%%%%%%%%%%%%%%%%%%%%%%%%%%%%%%%%%%%%%%%%%%%%%%%%%%%%%%%%%%%%%%%%%%%%%%%%%%%%%%%%%%%%%%%%%%%%%%%%%%%%%%%%%%%%%%%%%%%%%%%%%%%%%%%%%%%%%%%%%%%%%%%%%%%%%%%%%%%%%%%%%%%

\section{Profile decomposition}\label{sec-prof}
In this section we introduce a profile decomposition for radial data.

\begin{lem}\label{profile}
	Assume that $\{v_{0, n}\} \subset \hr$, $\|v_{0, n}\|_{\h}\le A$, and $\|e^{it\Delta}v_{0,n}\|_{L_t^qL_x^r}\ge \de > 0$, where $\delta$ is as in Proposition \ref{lwp}. Then up to a subsequence $(\mbox{still called}\; \{v_{0, n}\})$ for any $J \ge 1$ there exists a sequence $\{V_{0,j}\}_{1 \le j \le J}$ in $\hr$ and a family of parameters $(\ljn,\tjn)\in \mathbb{R}^+\times\mathbb{R}$ with
	$$
	\frac{\ljn}{\lam_{j',n}} + \frac{\lam_{j',n}}{\ljn} +\frac{|\tjn-t_{j',n}|}{\ljn^2} \xrightarrow{n\to\infty} 0 \quad  j\neq j'
	$$
	such that
	\begin{enumerate}
		\item[$(i)$]   $\|V_{0,1}\|_{\h}\ge \al_0(A) > 0$,
		\item[$(ii)$]  $v_{0,n}=\sum_{j=1}^{J} \ljn^{-\frac12}V_j^l\left(-\frac{\tjn}{\ljn^2}, \frac{x}{\ljn}\right) + w_n^J, \quad V_j^l(t, x) = [e^{it\Delta}V_{0, j}](x)$,
		\item[$(iii)$] $\underset{n\to\infty}{\liminf}\left\|\nabla e^{it\Delta}w_n^J\right\|_{L_t^qL_x^r} \to 0$ for any admissible pair $(q, r)$ with $2 < q < \infty$,
		\item[$(iv)$]  $\|v_{0,n}\|_{\h}^2 = \sum_{j=1}^{J}\|V_{0,n}\|_{\h}^2 + \|w_n^J\|_{\h}^2 + o(1)$ as $n \to \infty$,
		\item[$(v)$]   $E_g(v_{0,n}) = \sum_{j=1}^{J}E_g\left(\ljn^{-\frac12}V_j^l(-\frac{\tjn}{\ljn^2}, \frac{\cdot}{\ljn})\right) + E_g(w_n^j)+o(1)$ as $n \to \infty$.
	\end{enumerate}
\end{lem}
\begin{rem}
	The space-frequency translation has been removed by radial symmetry in the above lemma. This fact enables us to get a stronger convergence of remainder terms $w_n^J$ in $L_t^q \dot W^{1, r}$ than expected in $L_t^q L^\frac{3r}{3-r}$. Together with long-time perturbation (Proposition \ref{per}), we will utilize this strong convergence for the proof of the existence and compactness of minimal energy blowup solution. For the non-radial case one has only to replace the convergence in the space $L_t^q L^\frac{3r}{3-r}$.
\end{rem}

\begin{proof}
	One can show $(i), (ii), (iii)$, and $(iv)$ by exactly the same way as in \cite{ke,chhwoz,chhwkwle}. We omit the details. Here we only consider the energy decoupling $(v)$.
	Due to the kinetic energy decoupling $(iv)$ it suffices to show
	$$
	\lim_{n \to \infty} \left(\int g|v_{0, n}|^{p+1}\,dx - \sum_{1 \le j \le J}\int g|G_{j, n}|^{p+1}\,dx - \int g|w_n^J|^{p+1}\,dx\right) = 0,
	$$
	where
	$$
	G_{j,n}(x) = \ljn^{-\frac12}V_j^l\left(-\frac{\tjn}{\ljn^2}, \frac{x}{\ljn}\right).
	$$
	For this we consider
	\begin{align}\label{1st-step}
	\int g(x)|v_{0,n}|^{p+1} dx - \int g(x)|v_{0,n} - G_{1,n}|^{p+1}\,dx -\int g(x)|G_{1,n}|^{p+1} dx \xrightarrow{n\to\infty} 0.
	\end{align}
	Then by repeating the same argument w.r.t. $j$ we conclude the proof.

In order to show \eqref{1st-step} we utilize the following inequality: (p.358 of \cite{lieb})
\begin{align}\label{diff-ineq}
\Big||z_1|^q - |z_1-z_2|^q - |z_2|^q\Big| \le q2^{q-1}\Big(|z_1-z_2|^{q-1}|z_2| + |z_1-z_2||z_2|^{q-1}\Big)\;\;\mbox{for}\;\;1 \le q < \infty.
\end{align}

	Let $s_n = -\frac{t_{1, n}}{\lambda_{1, n}^2}$. Suppose that $\underset{n \to \infty}{\lim}|s_n|= \infty$. Then by the time-decay estimate such that $\|e^{it_n\Delta}f\|_{L_x^p} \to 0$ as $|t_n| \to \infty$ for any $p > 2$ and $f \in C_0^\infty$, we have from \eqref{diff-ineq} that
	\begin{align*}
	&\int g(x)|v_{0,n}|^{p+1} dx - \int g(x)|v_{0,n} - G_{1,n}|^{p+1}\,dx -\int g(x)|G_{1,n}|^{p+1} dx\\
	&\qquad  \les \int |x|^{-b}(|v_{0,n}-G_{1,n}|^p|G_{1,n}| + |v_{0,n}-G_{1,n}||G_{1,n}|^p) dx\\
	&\qquad  \les \||x|^{-1}(v_{0,n} - G_{1, n})\|_{L_x^2}^b\|v_{0, n} - G_{1,n}\|_{L_x^6}^{p-b}\|G_{1, n}\|_{L_x^6} + \||x|^{-1}G_{1, n}\|_{L_x^2}^b\|v_{0,n} - G_{1, n}\|_{L_x^6}\|G_{1,n}\|_{L_x^6}^{p-b}\\
&\qquad \to 0 \;\;\mbox{as}\;\;n \to \infty.
	\end{align*}

	Up to a subsequence we now assume that $\underset{n \to \infty}{\lim}s_n= \bar{t} $. And let $\sca(v_{0,n})= \ljn^{\frac12}v_{0,n}(\ljn\cdot)$.
	Then we have that
	\begin{align}
	&\sca(v_{0,n}) \to e^{i\bar{t}\Delta}V_{0,1} \quad \text{weakly in} ~\h ~\text{as}~ n \to \infty,\nonumber\\
	&\sca(G_{1,n}) \to e^{i\bar{t}\Delta}V_{0,1} \quad \text{strongly in} ~L_x^6 ~\text{as}~ n \to \infty.\label{strong-con}
	\end{align}
	By scaling the left hand side of \eqref{1st-step} is written as
	\begin{align*}
	&\int g(x)|v_{0,n}|^{p+1} dx - \int g(x)|v_{0,n}  - G_{1,n}|^{p+1}\,dx -\int g(x)|G_{1,n}|^{p+1} dx \\
	&\quad = \int (\ljn)^b g(\ljn x)|\sca(v_{0,n})|^{p+1} dx - \int (\ljn)^b g(\ljn x)|\sca(v_{0,n}) -\sca(G_{1,n})|^{p+1} dx\\
	&\qquad\qquad - \int (\ljn)^b g(\ljn x)\label{key}|\sca(G_{1,n})|^{p+1} dx\\
	&\quad =: A_n + B_n + C_n,
	\end{align*}
	where
	\begin{align*}
	A_n &= \int (\ljn)^b g(\ljn x)|\sca(v_{0,n})|^{p+1} dx - \int (\ljn)^b g(\ljn x)|\sca(v_{0,n})-e^{i\bar{t}\Delta}V_{0,1}|^{p+1}dx\\
&\qquad\qquad\qquad\qquad\qquad\qquad\qquad\quad\;\;\;\; - \int (\ljn)^b g(\ljn x)|e^{i\bar{t}\Delta}V_{0,1}|^{p+1}\,dx,\\
	B_n &= \int (\ljn)^b g(\ljn x)|\sca(v_{0,n})-e^{i\bar{t}\Delta}V_{0,1}|^{p+1}\,dx - \int (\ljn)^b g(\ljn x)|\sca(v_{0,n}) -\sca(G_{1,n})|^{p+1}\, dx,\\
	C_n &= \int (\ljn)^b g(\ljn x)|e^{i\bar{t}\Delta}V_{0,1}|\,dx - \int (\ljn)^b g(\ljn x)|\sca(G_{1,n})|\,dx.
	\end{align*}

	Using the density by $C_0^\infty$, and the compactness of multiplication operator by $\beta \in C_0^\infty$ from $\h$ to $L^3$, one can readily show that for any $\varepsilon > 0$ there exist $\beta \in C_0^\infty$ close to $e^{i\bar{t}\Delta}V_{0,1}$ in $\h$ and $N = N(\varepsilon)$ such that if $n \ge N$, then $(\sca(v_{0,n})- e^{i\bar{t}\Delta}V_{0,1})\beta$ close to $0$ in $L_x^3$ and thus
	\begin{align*}
	|A_n| &\le \int (\ljn)^b|g(\ljn x)|\left|\sca(v_{0,n})|^{p+1} -|\sca(v_{0,n})-e^{i\bar{t}\Delta}V_{0,1}|^{p+1}- |e^{i\bar{t}\Delta}V_{0,1}|^{p+1}\right|\,dx\\
	&\les \int |x|^{-b}\left||\sca(v_{0,n})|^{p+1} -|\sca(v_{0,n})-e^{i\bar{t}\Delta}V_{0,1}|^{p+1}- |e^{i\bar{t}\Delta}V_{0,1}|^{p+1}\right|\,dx\\
	& \les \int |x|^{-b}\left( |\sca(v_{0,n})-e^{i\bar{t}\Delta}V_{0,1}|^p|e^{i\bar{t}\Delta}V_{0,1}| + |\sca(v_{0,n})-e^{i\bar{t}\Delta}V_{0,1}||e^{i\bar{t}\Delta}V_{0,1}|^p\right)\, dx\\
	& \les \|\sca(v_{0,n})-e^{i\bar{t}\Delta}V_{0,1}\|_{\h}^b\|\sca(v_{0,n})-e^{i\bar{t}\Delta}V_{0,1}\|_{L_x^6}^{p-b}\|e^{i\bar{t}\Delta}V_{0,1}-\beta\|_{L_x^6}\\
	& \qquad +\|(\sca(v_{0,n})-e^{i\bar{t}\Delta}V_{0,1})\beta\|_{L_x^3}\||x|^{-1}(\sca(v_{0,n})-e^{i\bar{t}\Delta}V_{0,1})\|_{L_x^2}^b\|\sca(v_{0,n})-e^{i\bar{t}\Delta}V_{0,1}\|_{L_x^6}^{p-1-b}\\
	& \qquad +\|\sca(v_{0,n})-e^{i\bar{t}\Delta}V_{0,1}\|_{L_x^6}\||x|^{-1}(e^{i\bar{t}\Delta}V_{0,1}-\beta)\|_{L_x^2}^b\|e^{i\bar{t}\Delta}V_{0,1}-\beta\|_{L_x^6}^{p-b}\\
	& \qquad +\|(\sca(v_{0,n})-e^{i\bar{t}\Delta}V_{0,1})\beta\|_{L_x^3}\||x|^{-1}\beta\|_{L_x^2}^{b}\|\beta\|_{L_x^6}^{p-1-b}\\
    &  < \varepsilon.
	\end{align*}
We used \eqref{diff-ineq} for the third inequality.
	
	On the other hand, we have by direct calculation that
	\begin{align*}
	|B_n| &\le \int (\ljn)^b|g(\ljn x)|\left||\sca(v_{0,n}) -\sca(G_{1,n})|^{p+1} - |\sca(v_{0,n}) - e^{i\bar{t}\Delta}V_{0,1}|^{p+1}\right|\,dx\\
	&\les   \int|x|^{-b}\left(|\sca(v_{0,n})|^p + |\sca(G_{1,n})|^{p} + |e^{i\bar{t}\Delta}V_{0,1}|^{p}\right)|\sca(G_{1,n}) - e^{i\bar{t}\Delta}V_{0,1}|\,dx,\\
	|C_n| &\le  \int (\ljn)^b|g(\ljn x)|\left||\sca(G_{1,n})|^{p+1}  - |e^{i\bar{t}\Delta}V_{0,1}|^{p+1}\right|\,dx\\
	&\les  \int|x|^{-b} \left(|\sca(G_{1,n})|^{p}  + |e^{i\bar{t}\Delta}V_{0,1}|^{p}\right)|\sca(G_{1,n}) - e^{i\bar{t}\Delta}V_{0,1}|\,dx.
	\end{align*}
	Thus the strong convergence \eqref{strong-con} concludes that
	$$\lim_{n\to\infty}(B_n + C_n) = 0.$$
	This completes the proof of Lemma \ref{profile}.
\end{proof}

%%%%%%%%%%%%%%%%%%%%%%%%%%%%%%%%%%%%%%%%%%%%%%%%%%%%%%%%%%%%%%%%%%%%%%%%%%%%%%%%%%%%%%%%%%%%%%%%%%%%%%%%%%%%%%%%%%%%%%%%%%%%%%%%%%%%%%%%%%%%%%%%%%%%%%%%%%%%%%%%%%%%%%%%%%%%%%%%%%%%%%%%%%%%%%%%%%%%%%%%%%%%%%%%%%%%%%%%%%%%%%%%%%%%%%%%%%%%%%%%%%%%%%%%%%%%%%%%%%%%%%%%%%%%%%%%%%%%%%%%%%%%%%%%%%%%%%%%%%%%%%%%%%%%%%%%%%%%%%%%%%%%%%%%%%%%%%%%%%%%%%%%%%%%%%%%%%%%%%%%%%%%%%%%%%%%%%%%%%%%%%%%%%%%%%%%%%%%%%%%%%%%%%%%%%%%%%%%%%%%%%%%%%%%%%%%%%%%%%%%%%%%%%%%%%%%%%%%%%%%%%%%%%%%%%%%%%%%%%%%%%%%%%%%%%%%%%%%%%%%%%%%%%%%%%%%%%%%%%%%%%%%%%%%%%%%%%%%%%%%%%%%%%
\section{Minimal energy blowup solution (MEBS)}\label{sec-mebs}
The aim of this section is to show the existence and compactness of MEBS. Our proof is quite similar to that of \cite{km} except for the inhomogeneous coefficient $g$.
%The proof of inhomogeneous type is desctibed in \cite{chkl2}.
However, for the convenience of readers, we provide a sketch of proof.
\subsection{Existence of MEBS}

For each $0 < e < \eg(Q_b)$ let
$$
\mathcal A(e) := \left\{\varphi \in \hr : E_g(\varphi) < e, \gs\|\varphi\|_{\h}^2 < \|Q_b\|_{\h}^2 \right\}$$
and then
$$\mathcal \beta(e) := \sup\Big\{\|v\|_{\sst} : v(0) \in \mathcal A(e), v \;\;\mbox{solution to}\;\;\eqref{maineq} \Big\}.
$$
Define $\egc = \sup\{e : \mathcal \beta(e) < +\infty\}$. In view of the blowup criterion and small data scattering (Remarks \ref{blowup-cri} and \ref{scattering}) we deduce that $0 < \egc \le E_g(Q_b)$. In this subsection we assume that $\egc < \eg(Q_b)$, which will lead us to a contradiction.

By the definition of $\egc$ we deduce that
\begin{enumerate}
	\item If $0 \le e <\egc$, $\gs\|\varphi\|_{\h}^2 < \|Q_b\|_{\h}^2$, and $\eg(\varphi) < e$, then $\|u\|_{\sst} < +\infty$.
	
	\item If $\egc \le e < \eg(Q_b)$, $\gs\|\varphi\|_{\h}^2 < \|Q_b\|_{\h}^2$, and  $\egc \le \eg(\varphi)< e <\eg(Q_b)$, then $\|u\|_{\sst} = +\infty$.
\end{enumerate}
At this point, we may expect that $\eg(Q_b)$ is critical value between GWP and blowup.

\begin{prop}\label{mebs}
	Let $\varphi_c \in \hr$ satisfy that $E_g(\varphi_c) = \egc(<E_g(Q_b))$ and $\gs\|\varphi_c\|_{\h}^2 < \|Q_b\|_{\h}^2$. If $u_c$ is the corresponding solution to \eqref{maineq}, then $\|u_c\|_{\sst} = +\infty$.
\end{prop}
The solution $u_c$ is called the minimal energy blowup solution (MEBS). For the proof we need to introduce nonlinear profile.

\begin{defn}[Nonlinear profile]\label{nl-pro}
	Let $v_0 \in \hr$, $v = e^{it\Delta}v_0$, and $\{t_n\}$ a sequence with $t_n \to \bar{t} \in [-\infty,+\infty]$. We say that $u(t, x)$ is a non-linear profile associated with $(v_0,\{t_n\})$ if there exist a maximal interval $I^*$ with $\bar{t} \in I^*$ such that $u$ is solution of \eqref{maineq} on $I^*$ and $\underset{n\to\infty}{\lim}\|u(\cdot,t_n)-v(\cdot,t_n)\|_{\h} = 0$.
\end{defn}

\begin{rem}
	The nonlinear profile always exists. In fact, if $\bar{t} \in (-\infty,+\infty)$, we have only to solve \eqref{maineq} with initial data $v(x,\bar{t})$. And if $\bar{t} = \pm\infty$, by solving the integral equation
	$$
	u(t) = e^{it\Delta}v_0 +i\int_{t}^{\pm\infty}e^{i(t-t')\Delta}[g|u|^{p-1}u]dt'
	$$
	we get the nonlinear profile.
\end{rem}

\begin{rem}
	If $u_1$ and $u_2$ are nonlinear profiles associated with $(v_0,\{t_n\})$, then $u_1 = u_2$ by LWP. Hence the uniqueness of nonlinear profile is guaranteed.
\end{rem}

%%%%%%%%%%%%%%%%%%%%%%%%%%%%%%%%%%%%%%%%%%%%%%%%%%%%%%%%%%%%%%%%%%%%%%%%%%%%%%%%%%%%%%%%%%%%%%%%%%%%%%%%%%%%%%%%%%%%%%%%%%%%%%%%%%%%%%%%%%%%%%%%%%%%%%%%%%%%%%%%%%%%%%%%%%%%%%%%%%%%%%%%%%%%%%%%%%%%%%%%%%%%%%%%%%%%%%%%%%%%%%%%%%%%%%%%%%%%%%%%%%%%%%%%%%%%%%%%%%%%%%%%%%%%%%%%%%%%%%%%%%%%%%%%%%%%%%%%%%%%%%%%%%%%%%%%%%%%%%%%%%%%%%%%%%%%%%%%%%%%%%%%%%%%%%%%%%%%%%%%%%%%%%%%%%%%%%%%%%%%%%%%%%%%%%%%%%%%%%%%%%%%%%%%%%%%%%%%%%%%%%%%%%%%%%%%%%%%%%%%%%%%

The following lemma is useful for the proof of Propositions \ref{mebs} and \ref{compactness}. While proving this lemma, we demonstrate the usage of profile decomposition and long-time perturbation.

\begin{lem}\label{en-con}
	Let $\{z_{0,n}\}\in\hr$, $\gs\|z_{0,n}\|_{\h}^2 < \|Q_b\|_{\h}^2$, $E_g(z_{0,n}) \to E_{g,c}(< E_g(Q_b))$, and $\|e^{it\Delta}z_{0,n}\|_{\s}\ge \delta > 0$. Let $\{V_{0,j}\}$ be the linear profiles.	Assume that one of
	\begin{enumerate}
		\item[$(a)$] $\underset{n\to\infty}{\liminf} E_g(V_1^l(-\frac{t_{1,n}}{\lam_{1,n}^2})) < E_{g,c}$,
		\item[$(b)$] $\underset{n\to\infty}{\liminf} E_g(V_1^l(-\frac{t_{1,n}}{\lam_{1,n}^2})) = E_{g,c}$ and for $s_n=-\frac{t_{1,n}}{\lam_{1,n}^2},$ after passing to a subsequence so that $s_n \to \bar{s}\in[-\infty,\infty]$ and $E_g(V_1^l(s_n)) \to E_{g,c}$, and if $U_1$ is the nonlinear profile associated with $(V_{0,1},\{s_n\})$, then $I = \mathbb R$ and $\|U_1\|_{\s} < \infty$.
	\end{enumerate}
	Then, if $\{z_n\}$ is solution of \eqref{maineq} with $\{z_{0,n}\}$, $\|z_n\|_{\s} < \infty$, for $n$ large.
\end{lem}

\begin{proof}
	
	At first we show part $(b)$. Assume that
	$$
	\underset{n\to\infty}{\liminf} E_g(V_1^l(-\frac{t_{1,n}}{\lam_{1,n}^2})) = E_{g,c}.
	$$
	From the energy coercivity, $(iii)$ of Lemma \ref{trapping}, and the profile decomposition it follows that $V_{0,j} = 0\; (j \ge 2)$ and $\|w_n^J\|_{\h} \to 0$. Now let us set
	$$
	v_{0,n}(x)=\lam_{1,n}^{\frac12}z_{0,n}(\lam_{1,n} x),\quad \tilde{w}_n^J(x)= \lam_{1,n}^{\frac12}w_n^J(\lam_{1,n} x).
	$$
	Then we have
	$$
	v_{0,n}= V_1^l(s_n) + \tilde{w}_n^J, \quad \|\tilde{w}_n^J\|_{\h} \to 0,
	$$
	with $\gs\|v_{0,n}\|_{\h}^2 < \|Q_b\|_{\h}^2$ and $E_g(v_{0,n}) \to E_{g,c} < E_g(Q_b)$.
	
	Then by definition of linear and nonlinear profile we get
	\begin{align*}
	&\|V_1^l(s_n) - U_1(s_n)\|_{\h} \to 0,\\
	& v_{0,n} = U_1(s_n) + \bar{w}_n^j,\;\; \|\bar{w}_n^j\|_{\h} \to 0,\\
	&\|\nabla e^{it\Delta}[V_1^l(s_n) - U_1(s_n)]\|_{W_i(I)} \to 0.
	\end{align*}
	Since $\|\bar{w}_n^j\|_{\h} \to 0$, we have $E_g(U_1)= E_{g,c} < E_g(Q_b)$ and also from $(i)$ of Lemma \ref{trapping} that
	\begin{align*}
	\underset{t\in I}{\sup}\;\gs \|U_1(t)\|_{\h}^2 < \|Q_b\|_{\h}^2.
	\end{align*}
	Therefore the case $(b)$ follows from the long-time perturbation, Proposition \ref{per}.
	
	We next assume that
	$$\underset{n\to\infty}{\liminf} E_g(V_1^l(-\frac{t_{1,n}}{\lam_{1,n}^2})) = E_{g,c}.$$
	We will show that
	$$\underset{n\to\infty}{\liminf} E_g(V_j^l(-\frac{t_{1,n}}{\lam_{1,n}^2})) < E_{g,c} \quad(j=2,\cdots,J).$$
	By profile decomposition, Lemma \ref{profile} we have
	\begin{align*}
	&\|z_{0,n}\|_{\h}^2 = \sum_{j=1}^{J}\|V_{0,j}\|_{\h}^2 + \|w_n^J\|_{\h}^2 + o(1) ~\text{as}~  n \to \infty,&\\
	&E_g(z_{0,n}) \to E_{g,c} < E_g(Q_b).&
	\end{align*}
	For $n$ large, there exists $\delta_0 > 0$ such that $E_g(z_{0,n})\le(1-\delta_0)E_g(Q_b)$ and thus $(i)$ of Lemma \ref{trapping} gives us that $\gs\|z_{0,n}\|_{\h}^2 \le (1- \bar{\delta}) \|Q_b\|_{\h}^2$. Then we obtain that $E_g(V_j^l(s_n)) \ge 0$ and $ E_g(w_n^J) \ge 0$ for all $n$ large, and also that $E_g(V_1^l(s_n)) \ge C\al_0 >0$ for some $\alpha_0 > 0$. Thus,
	$$\eg(z_{0,n}) \ge C\al_0 + \sum_{j=2}^{J}\eg(V_j^l(s_n)) + o(1).$$
	Since $\eg(z_{0,n}) \to \egc$,
	$$\underset{n\to\infty}{\liminf} E_g(V_j^l(s_n)) < E_{g,c}\;\;\mbox{for}\;\; j \ge 2.$$
	
	Let $U_j$ be the nonlinear profile associated with $(V_{0,j},\{s_n\})$. Our next claim is that $\|U_j\|_{S(\mathbb{R})} < +\infty$ for all $j = 1,2, \cdots,J$. By Definition \ref{nl-pro} we have
	$$
	\eg(V_j^l(s_n))<\egc \quad\eg(U_j) < \egc,
	$$
	and
	$$
	\gs\|V_j^l(s_n)\|_{\h}^2 \le \gs\|z_{0,n}\|_{\h}^2 +o(1) \le (1-\bar{\delta})\|Q_b\|_{\h}^2 + o(1)
	$$
	for $n$ large.
	Hence $(i)$ of Lemma \ref{trapping} shows $\gs\|U_j(t)\|_{\h}^2 < \|Q_b\|_{\h}^2$ for all $t\in I_j$. But the criticality of $\egc$ means $I_j^*=\mathbb{R}$. Hence, $\|U_j\|_{S(\mathbb{R})} < \infty$. In fact, for fixed $J$ and large $n$ we have
	$$\sum_{j=1}^{J}\gs\|V_{0,j}\|_{\h}^2 \le \gs\|z_{0,n}\|_{\h}^2 +o(1) \le 2\|Q_b\|_{\h}^2.$$
	Then there exists $j_0$ such that for $j \ge j_0$, $\|V_{0,j}\|_{\h} \le \tilde{\delta}$ with $\tilde \delta > 0$ so small that $\|e^{it\Delta}V_{0,j}\|_{S(\mathbb{R})}\le \delta$. This shows that
	\begin{align*}
	\|U_j\|_{S(\mathbb{R})} &\le 2\delta,\\
	\|U_j\|_{L_t^\infty\h} + \|\nabla U_j\|_{W_1(\mathbb{R})} &+ \|\nabla U_j\|_{W_2(\mathbb{R})} \le C\|V_{0,j}\|_{\h}.
	\end{align*}
	Therefore we get $\|U_j\|_{S(\mathbb{R})} \le C \|V_{0,j}\|_{\h}$ for $j \ge j_0$.
	
	Next we define, for $\varepsilon_0 > 0$ to be chosen later,
	$$\hn = \sum_{j=1}^{J(\varepsilon_0)} (\lam_{j,n})^{-\frac12}U_j\left( \frac{t-t_{j,n}}{\lam_{j,n}^2}, \frac{x}{\lam_{j,n}} \right).$$
	Then from the density of $U_j$ by $C_0^\infty(\mathbb R^{1+3})$ functions and the orthogonality of $(\lam_{j,n}, t_{j,n})$ one can readily obtain that there exists constant $M > 0$ such that $\|\hn\|_{S(\mathbb{R})} \le M$ uniformly in $\varepsilon_0$ for $n \ge n(\varepsilon_0)$.
	%\begin{align*}
	%\left\|\hn\right\|_{S(\mathbb{R})}^{10} &\le \sum_{j=1}^{J(\varepsilon_0)}\iint \left| (\lam_{j,n})^{-\frac12}U_j\left( \frac{x}{\lam_{j,n}}, \frac{t-t_{j,n}}{\lam_{j,n}^2}\right)\right|^{10}\\
	%&\quad+ C_J \sum_{j\neq j'} \iint \left| (\lam_{j,n})^{-\frac12}U_j\left( \frac{x}{\lam_{j,n}}, \frac{t-t_{j,n}}{\lam_{j,n}^2}\right)\right|\cdot\left| (\lam_{j',n})^{-\frac12}U_{j'}\left( \frac{x}{\lam_{j',n}}, \frac{t-t_{j',n}}{\lam_{j',n}^2}\right)\right|^9\\
	%&= A + B.
	%\end{align*}
	%By , $A \to 0$. And
	%\begin{align*}
	%B ~&\le~ \sum_{j=1}^{j_0}\|U_j\|_{S(\mathbb{R})}^{10} + \sum_{j=j_0+1}^{J(\varepsilon_0)}\|U_j\|_{S(\mathbb{R})}^{10}\\
	%  &\le~ \sum_{j=1}^{j_0}\|U_j\|_{S(\mathbb{R})}^{10} + C\sum_{j=j_0+1}^{J(\varepsilon_0)}\|V_{0,j}\|_{\h}^{10}\\
	% &\le~ \sum_{j=1}^{j_0}\|U_j\|_{S(\mathbb{R})}^{10} + C\left(\sup_{j\ge j_0+1}\|V_{0,j}\|_{\h}^8\right)\sum_{j=j_0+1}^{J(\varepsilon_0)}\|V_{0,j}\|_{\h}^{2}\\
	%&\le~ \frac{C_0}{2}.
	%\end{align*}

	Let us now define
	$$
	\rne = g|\hn|^{p-1}\hn - \sum_{j=1}^{J}g|\ujn|^{p-1}\ujn,
	$$
	where
	$$
	\ujn(t, x) = \lam_{j,n}^{-\frac12}U_j\left(\frac{t-t_{j,n}}{\lam_{j,n}^2}, \frac{x}{\lam_{j,n}}\right).
	$$
	Using the Lemmas \ref{str}, \ref{hs-ineq} and definition of $\rne$, $\|\nabla\rne\|_{L_t^2L_x^{\frac65}} \to 0$ as $n \to \infty$.
	Since $\|U_j\|_{S(\mathbb{R})} < \infty$, we deduce that $\|\nabla U_j\|_{W_i(\mathbb{R})} < \infty$. To apply Proposition \ref{per} let us set $\tilde{u} = \hn, e = \rne$ and choose $J(\varepsilon_0)$ large that for $n$ large $\|\nabla e^{it\Delta}w_n^{J(\varepsilon_0)}\|_{W_i(\mathbb{R})} \le \frac{\varepsilon_0}{2}\;(i=1,2)$. Then, for $n$ large,
	$$
	z_{0,n} = \hn(0) + \tilde{w}_n^{J(\varepsilon_0)},
	$$
	where $\|\nabla e^{it\Delta}\tilde{w}_n^{J(\varepsilon_0)}\|_{W_i(\mathbb{R})}\le \varepsilon_0.$ We can show that $\|\nabla\hn\|_{L_t^{\infty}\h} \le A$ uniformly in $\varepsilon_0$ by the same way as of $\hn$.
	Then for $n$ large, $\|\tilde{w}_n^{J(\varepsilon_0)}\|_{\h} \le 2\|Q_b\|_{\h} =: A'$. Choose $\varepsilon < \varepsilon_0 = \varepsilon_0(M, A, A')$. Then
	the long-time perturbation, Proposition \ref{per} leads us to the case $(a)$.
\end{proof}

\vspace{5mm}
%%%%%%%%%%%%%%%%%%%%%%%%%%%%%%%%%%%%%%%%%%%%%%%%%%%%%%%%%%%%%%%%%%%%%%%%%%%%%%%%%%%%%%%%%%%%%%%%%%%%%%%%%%%%%%%%%%%%%%%%%%%%%%%%%%%%%%%%%%%%%%%%%%%%%%%%%%%%%%%%%%%%%%%%%%%%%%%%%%%%%%%%%%%%%%%%%%%%%%%%%%%%%%%%%%%%%%%%%%%%%%%%%%%%%%%%%%%%%%%%%%%%%%%%%%%%%%%%%%%%%%%%%%%%%%%%%%%%%%%%%%%%%%%%%%%%%%%%%%%%%%%%%%%%%%%%%%%%%%%%%%%%%%%%%%%%%%%%%%%%%%%%%%%%%%%%%%%%%%%%%%%%%%%%%%%%%%%%%%%%%%%%%%%%%%%%%%%%%%%%%%%%%%%%%%%%%%%%%%%%%%%%%%%%%%%%%%%%%%%%%%%%
Now we prove the Proposition \ref{mebs}.

\begin{proof}[Proof of Proposition \ref{mebs}]
	Let $u_{0,n}\in\hr$ be such that
	$$
	\gs\|u_{0,n}\|_{\h}^2 < \|Q_b\|_{\h}^2, ~E_g(u_{0,n}) \to E_{g,c},~ \|e^{it\Delta}u_{0,n}\|_{S(I_n^*)}\ge\delta, ~\text{and}~ \|u_n\|_{S(I_n^*)} = +\infty,
	$$
	where $I_n^*$ are maximal intervals. Since $E_{g,c} < E_g(Q_b)$, we deduce that $E_g(u_{0,n})\le (1-\delta_0)E_g(Q_b)$ for large $n$.
	Using the Lemma \ref{trapping}, we can find a $\bar{\delta}$ such that
	$$
	\gs\|u_n(t)\|_{\h}^2 \le (1-\bar{\delta})\|Q_b\|_{\h}^2\;\;\mbox{for all}\;\;t \in I_n^*.
	$$
	We fix $J \ge 1$ and apply the profile decomposition to $\{u_{0,n}\}$. Then we get
	\begin{eqnarray}
	\|u_{0,n}\|_{\h}^2 = \sum_{j=1}^{J}\|V_{0,j}\|_{\h}^2 + \|w_n^J\|_{\h}^2 + o(1),\label{ki-decom}\\
	E_g(u_{0,n}) = \sum_{j=1}^{J}E_g(V_j^l(s_n)) + E_g(w_n^J) + o(1),\label{en-decom}
	\end{eqnarray}
	where $s_n=-\frac{t_{1,n}}{\lam_{1,n}^2}$.
	
	For $n$ large, $\gs\|w_n^J\|_{\h}^2 \le (1- \frac{\bar{\delta}}{2})\|Q_b\|_{\h}^2$ and  $\gs\|V_{0,j}\|_{\h}^2 \le (1- \frac{\bar{\delta}}{2})\|Q_b\|_{\h}^2$. These imply that $E_g(w_n^J) \ge 0, E_g(V_j^l(s_n)) \ge 0$ if $n$ is large. Thus we get $E_g(V_1^l(s_n))\le E_g(u_{0,n}) +o(1)$ and then
	$$
	\underset{n\to\infty}{\liminf}E_g(V_1^l(s_n)) \le E_c.
	$$
	
	We assume that $\underset{n\to\infty}{\liminf}E_g(V_1^l(s_n)) < E_c.$ Then this contradicts $(a)$ of Lemma \ref{en-con}. Therefore,
	$$
	\underset{n\to\infty}{\liminf}E_g(V_1^l(s_n)) = E_c.
	$$
	
	Let $U_1$ be the non-linear profile associated with $(V_{0,1},\{s_n\}).$ By \eqref{en-decom} and the facts that
	$$
	E_g(u_{0,n}) \to E_{g,c}\;\;\mbox{and}\;\; E_g(V_1^l(s_n)) \to E_{g,c},
	$$
	we see that $E_g(w_n^J) \to 0, E_g(V_j^l(s_n))\to 0$ for $j \ge 2$.
	
	But the energy coercivity ($(iii)$ of Lemma \ref{trapping}) shows that
	$$
	\sum_{j=2}^{J}\|V_j^l(s_n)\|_{\h}^2 + \|w_n^j\|_{\h}^2 \xrightarrow{n \to \infty} 0.
	$$
	Since $\|V_j^l(s_n)\|_{\h} = \|V_{0,j}\|_{\h} $, $V_{0,j}=0(j>2)$, and $\|w_n^J\|_{\h}^2 \to 0$, so that
	$$
	u_{0,n} = \lam_{1,n}^{-\frac12}V_1^l\left(s_n,\frac{x}{\lam_{1,n}}\right)+ w_n^J.
	$$
	Rescaling as $v_{0,n}(x) = \lam_{1,n}^\frac12u_{0,n}\left(\lam_{1,n}x\right) $, we have $v_{0,n}=V_0^l(s_n)+\tilde{w}_n^J$ with $\|\tilde{w}_n^J\|_{\h} \to 0.$ By definition of nonlinear profile, we get
	\begin{align*}
	E_g(U_1(s_n)) &= E_g(V_1^l(s_n)) +o(1) = E_{g,c} +o(1),\\
	\gs\|U_1(s_n)\|_{\h}^2 &= \gs\|V_1^l(s_n)\|_{\h}^2 +o(1) = \gs\|V_{0,1}\|_{\h}^2 +o(1)\\
	&= \gs\|u_{0,n}\|_{\h}^2 +o(1)   < \|Q_b\|_{\h}^2 \quad \text{as} \quad n\to\infty.
	\end{align*}
	For a fixed $\bar{s} \in I_1^*$, which $I_1^*$ is maximal interval of $U_1$, the energy conservation yields
	$$
	E_g(U_1(\bar{s})) = E_g(U_1(s_n)) \to E_{g,c}
	$$
	and hence $$E_g(U_1(\bar{s})) = E_{g,c}.$$
	From the energy trapping it follows that $\gs\|U_1(\bar{s})\|_{\h}^2 < \|Q_b\|_{\h}^2$. If $\|U_1\|_{S(I_1^*)} < \infty$, it is a contradiction to $(b)$ of Lemma \ref{en-con} .
	Therefore we deduce
	$$
	\|U_1\|_{S(I_1^*)} = \infty.
	$$
	By setting $U_1=u_c$ we conclude the proof of Proposition \ref{mebs}.
\end{proof}
\vspace{5mm}
%%%%%%%%%%%%%%%%%%%%%%%%%%%%%%%%%%%%%%%%%%%%%%%%%%%%%%%%%%%%%%%%%%%%%%%%%%%%%%%%%%%%%%%%%%%%%%%%%%%%%%%%%%%%%%%%%%%%%%%%%%%%%%%%%%%%%%%%%%%%%%%%%%%%%%%%%%%%%%%%%%%%%%%%%%%%%%%%%%%%%%%%%%%%%%%%%%%%%%%%%%%%%%%%%%%%%%%%%%%%%%%%%%%%%%%%%%%%%%%%%%%%%%%%%%%%%%%%%%%%%%%%%%%%%%%%%%%%%%%%%%%%%%%%%%%%%%%%%%%%%%%%%%%%%%%%%%%%%%%%%%%%%%%%%%%%%%%%%%%%%%%%%%%%%%%%%%%%%%%%%%%%%%%%%%%%%%%%%%%%%%%%%%%%%%%%%%%%%%%%%%%%%%%%%%%%%%%%%%%%%%%%%%%%%%%%%%%%%%%%%%%%
\subsection{Compactness of the MEBS flows}
\begin{prop}\label{compactness}
	For any $u_c$ as in Proposition \ref{mebs}, with $\|u_c\|_{\sst} = +\infty$, there exist $\lambda(t) \in \mathbb{R}^+$, $t \in I_+^* (I_+^* := I^* \cap \left[0,\infty \right))$ such that
	$$
	\mathcal M = \left\{ v(x,t) := \lambda(t)^{-\frac12} u_c\left(t, \frac{x}{\lambda(t)}\right) : t \in I_+^*  \right\}
	$$
	has compact closure in $\hr$.
\end{prop}

\begin{proof}
	If $\mathcal M$ does not have a compact closure in $\hr$, then there exists $\eta_0 > 0$ and $\{t_n\}$ with $t_n > 0$ such that for all $\lambda_0 \in \mathbb{R}^{+}$,
	$$
	\left\|(\lambda_0)^{-\frac12} u \left(t_n,\frac{x}{\lambda_0}  \right)-u(t_{n'},x)\right\|_{\h} \ge \eta_0 \quad (n \neq n'),
	$$
	where $u = u_c$. Passing through a subsequence (still called $\{t_n\} $), $\{t_n\}$ is assumed to have a limit $\bar{t} \in \overline{I_+^*} = [0,T_+]$. Then $\bar{t} = T_+$ by the continuity of flow.
	We now may assume that $\|e^{it\Delta}u(t_n)\|_{S(0,\infty)} \ge \delta$, where $\de$ is as in Proposition \ref{lwp}. Since
	\begin{align*}
	E_g(u(t))=E_g(\varphi)&=E_{g,c}<E_g(Q_b),\\
	\gs\|\varphi\|_{\h}^2 &< \|Q_b\|_{\h}^2,
	\end{align*}
	we get
	$$
	\gs\|u(t)\|_{\h}^2 \le (1- \bar{\delta})\|Q_b\|_{\h}^2
	$$
	for $t \in I_+^*$.

	Now we apply the profile decomposition to $v_{0,n} = u(t_n)$. In view of the proof of Proposition \ref{mebs}, we conclude that
	$$
	\underset{n\to\infty}{\liminf}E_g(V_1^l(s_n)) \le E_{g,c},
	$$
	where $s_n = -\frac{t_{1, n}}{\lambda_{1, n}^2}$.
	Since $\|e^{it\Delta}v_{0,n}\|_{S(0,\infty)}\ge \delta$, if $\underset{n\to\infty}{\liminf}E_g(V_1^l(s_n)) < E_{g,c}$, then $\|v_n\|_{S(0, \infty)} < \infty$.
	This contradicts the fact that $u$ is MEBS. Thus
	$$
	\underset{n\to\infty}{\liminf}E_g(V_1^l(s_n)) = E_{g,c},
	$$
	which implies that $V_{0,j}=0$ for $j\ge 2$ and $\|w_n^J\|_{\h} \to 0$.
 Therefore we deduce that
\begin{align}
u(t_n) = (\lam_{1,n})^{-\frac12}V_1^l\left(s_n,\frac{x}{\lam_{1,n}} \right)+ w_n^J.
\end{align}
	
	If $\left\{s_n\right\}$ is unbounded, then from the profile decomposition (5.3) one can show that there exists subsequences (still called $\{s_n\}, \{t_n\}$) such that
	$\|e^{it\Delta}u(t_n)\|_{S(0,\infty)} \le \de$ or $\|e^{it\Delta}u(t_n)\|_{S(-\infty, 0)} < \delta$ for large $n$. From the LWP it follows that
	$\|u\|_{S(t_n, +\infty)} \le 2\delta$ or $\|u\|_{S(-\infty, t_n)} \le 2\delta$, respectively. This contradicts the fact that $u$ is MEBS. Therefore $\left\{s_n\right\}$ should be bounded.

	By passing to a subsequence we may assume that $s_n \to t_0 \in \mathbb{R}$. Since $\|w_n^J\|_{\h} \to 0$, for arbitrary $\lambda_0 > 0$
	$$
	\left\| (\lam_0)^{-\frac12}(\lam_{1,n})^{-\frac12}V_1^l\left( s_n,\frac{x}{\lam_{1,n}\lam_0}\right) - (\lam_{1,n'})^{-\frac12}V_1^l\left( s_{n'},\frac{x}{\lam_{1,n'}}\right) \right\|_{\h} \ge \frac{\eta_0}{2} \;\,\mbox{if}\;\,n \neq n'.
	$$
	By the change of variables $x \mapsto \lam_{1,n'}y$ we get
	$$
	\left\| \left(\frac{\lam_{1,n'}}{\lam_0\lam_{1,n}}\right)^{-\frac12}V_1^l\left(s_n,\frac{\lam_{1,n'}y}{\lam_{1,n}\lam_0}\right) - V_1^l\left(s_{n'},y\right) \right\|_{\h} \ge \frac{\eta_0}{2} \;\,\mbox{if}\;\,n \neq n'.
	$$
	Since $s_n \to t_0$, by choosing $\lam_0 = \frac{\lam_{1,n'}}{\lam_{1,n}}$, we reach another contradiction. This completes the proof of Proposition \ref{compactness}.
\end{proof}

%%%%%%%%%%%%%%%%%%%%%%%%%%%%%%%%%%%%%%%%%%%%%%%%%%%%%%%%%%%%%%%%%%%%%%%%%%%%%%%%%%%%%%%%%%%%%%%%%%%%%%%%%%%%%%%%%%%%%%%%%%%%%%%%%%%%%%%%%%%%%%%%%%%%%%%%%%%%%%%%%%%%%%%%%%%%%%%%%%%%%%%%%%%%%%%%%%%%%%%%%%%%%%%%%%%%%%%%%%%%%%%%%%%%%%%%%%%%%%%%%%%%%%%%%%%%%%%%%%%%%%%%%%%%%%%%%%%%%%%%%%%%%%%%%%%%%%%%%%%%%%%%%%%%%%%%%%%%%%%%%%%%%%%%%%%%%%%%%%%%%%%%%%%%%%%%%%%%%%%%%%%%%%%%%%%%%%%%%%%%%%%%%%%%%%%%%%%%%%%%%%%%%%%%%%%%%%%%%%%%%%%%%%%%%%%%%%%%%%%%%%%%

\section{Rigidity theorem}

%%%%%%%%%%%%%%%%%%%%%%%%%%%%%%%%%%%%%%%%%%%%%%%%%%%%%%%%%%%%%%%%%%%%%%%%%%%%%%%%%%%%%%%%%%%%%%%%%%%%%%%%%%%%%%%%%%%%%%%%%%%%%%%%%%%%%%%%%%%%%%%%%%%%%%%%%%%%%%%%%%%%%%%%%%%%%%%%%%%%%%%%%%%%%%%%%%%%%%%%%%%%%%%%%%%%%%%%%%%%%%%%%%%%%%%%%%%%%%%%%%%%%%%%%%%%%%%%%%%%%%%%%%%%%%%%%%%%%%%%%%%%%%%%%%%%%%%%%%%%%%%%%%%%%%%%%%%%%%%%%%%%%%%%%%%%%%%%%%%%%%%%%%%%%%%%%%%%%%%%%%%%%%%%%%%%%%%%%%%%%%%%%%%%%%%%%%%%%%%%%%%%%%%%%%%%%%%%%%%%%%%%%%%%%%%%%%%%%%%%%%%%
In this section, we will remove the MEBS $u_c$ by rigidity theorem under the condition $\egc < \eg(Q_b)$.

\begin{prop}[Rigidity]\label{rigidity}
	Suppose that $g$ is nonnegative, bounded radial function satisfying the conditions \eqref{scaling}, \eqref{var-con}, and \eqref{rig-con}. Let $\vp \in \hr$ satisfy that $\eg(\vp) < \eg(Q_b)$ and $\gs\|\vp\|_{\h}^2 < \|Q_b\|_{\h}^2$.  Let $u$ be the corresponding solution to \eqref{maineq} with $\vp$ and let $I^* = (-T_-,T_+)$ be the maximal existence time interval. Assume there exists $\lam(t) > 0$ such that
	$$
	\mathcal M:= \left\{ v(t, x) = (\lam(t))^{-\frac12}u\left(t,\frac{x}{\lam(t)}\right) : t \in [0,T_+)\right\}
	$$
	has compact closure in $\hr$. %Assume one of
	%$$(i)~ T_+ < \infty,~ \lam(t) \ge \frac{C_0^{\frac12}}{(T_+-t)^{\frac12}} \quad \text{or}\quad (ii)~ T_+=\infty,~ \lam(t)\ge A_0 > 0.$$
	Then $T_+ = +\infty$ and $\vp = 0$.
\end{prop}
\begin{rem}\label{lam}
	If there exists a sequence $\{t_i\} \subset [0, T_+)$ such that $\lambda(t_i) \to 0$ as $i \to \infty$, then from the compactness of $\mathcal M$ it follows that there exists a subsequence $\{t_{i_j}\}$ such that $u(t_{i_j}, \cdot) = \lambda(t_{i_j})^\frac12 v(t_{i_j}, \lambda(t_{i_j})\cdot) \to 0$ as $j \to \infty$. This implies that $T_+  = +\infty$ and $\varphi = 0$. Therefore, we may assume that $\lambda(t) \ge A_0$ for some constant $A_0 > 0$.
\end{rem}

\begin{rem}\label{cpt}
	Assume that $\lambda(t) \ge A_0$. Then the scaling invariance and compactness of $\mathcal U$ bring us to the fact that for any $\varepsilon > 0$ there exists $r = r(\varepsilon)$ such that
	\begin{align*}
	\int_{|x| > r}\left(|\nabla u|^2 +|u|^{p+1} + \frac{|u|^2}{|x|^2}\right)dx &= \int_{|y| > r\lam(t)}\left( |\nabla v|^2 +|v|^{p+1}+\frac{|v|^2}{|x|^2}\right)\,dx\\
	&\le \int_{|y|\ge A_0r}(\cdots) \le \varepsilon
	\end{align*}
	for all $t \in I^*$.
\end{rem}
%We prove by contradiction. If $T_+<\infty$, we easily lead to a contradiction. If $T_+ = \infty$, for this purpose we show that $z_r'(t)$ is bounded and inequality such that
%$$z_r''(t)\ge 8 \left[ \underset{|x|\leq r}{\int}|\nabla u|^2 -g|u|^6\right] +\frac43 \underset{|x|\leq r}{\int} (x\cdot\nabla g)|u|^6 dx + [error]\ge \tilde{C}_{\de_0}\|\vp\|_{\h}^2$$
%through localized virial argument. See the \cite{q}.

The proof is divided into two cases: $T_+ < +\infty$ and $T_+ = +\infty$.
%%%%%%%%%%%%%%%%%%%%%%%%%%%%%%%%%%%%%%%%%%%%%%%%%%%%%%%%%%%%%%%%%%%%%%%%%%%%%%%%%%%%%%%%%%%%%%%%%%%%%%%%%%%%%%%%%%%%%%%%%%%%%%%%%%%%%%%%%%%%%%%%%%%%%%%%%%%%%%%%%%%%%%%%%%%%%%%%%%%%%%%%%%%%%%%%%%%%%%%%%%%%%%%%%%%%%%%%%%%%%%%%%%%%%%%%%%%%%%%%%%%%%%%%%%%%%%%%%%%%%%%%%%%%%%%%%
\subsection{Case: $T_+ < \infty$}
Suppose that $T_+ < +\infty$ and there exists a sequence $\{t_i\}$ such that $t_i \to T_+$ and $\lambda(t_i) \to \lambda_0 > 0$ as $i \to \infty$. Then by the compactness of $\mathcal U$, LWP, and long-time perturbation one can deduce that $\|u\|_{S(T_-+\delta, T_+-\delta)} < +\infty$ for some $\delta > 0$. This is a contradiction to the maximality of $T_+$ and hence implies that if $T_+ < +\infty$, then $\lambda(t) \to +\infty$ as $t \to T^+$. For details see p.667 of \cite{km}.

Now let $a(x)\in C_0^\infty(\rt)$ and $a_r(x)$ be as follows:
\begin{align*}a(x)&:=\left\{ \begin{array}{cc}
1 & (|x|\leq 1) \\
0 & (|x|\geq 10)
\end{array}  \right.,\quad
a_r(x):=a(\frac{x}{r}).
\end{align*}
And we define that
$$
y_r(t)= \int a_r|u(t)|^2dx
$$
for $t \in [0,T_+)$. Then the density by $H^2$ data yields
$$
y_r'(t) = 2 \mathrm{Im}\left( \int \bar{u}\nabla u \nabla a_r dx - \int a_rg|u|^{p+1} dx\right).
$$
Since $g$ is bounded, by Hardy-Sobolev inequality and energy trapping (Lemma \ref{trapping}) we have
\begin{align}\begin{aligned}\label{y'-bound}
|y_r'(t)| &\le C\left|\int \bar{u}\frac{1}{|x|} \nabla u \nabla a_r dx\right| + C\left|\int a_r|u|^{p+1} dx\right| \le C\|\nabla u\|_{L_x^2}\left\|\frac{u}{|x|}\right\|_{L_x^2} +C \|u\|_{L_x^{p+1}}^{p+1}\\
&\le C(\|Q_b\|_{\h}^2 +\|Q_b\|_{\h}^{p+1}).
\end{aligned}\end{align}
Next we will show that
\begin{align}\label{getL2}
\lim_{t \to T_+}\int_{|x|\le r}|u(t,x)|^2 dx = 0 \;\;\mbox{for all} \;\;r > 0.
\end{align}
Since $u(t, x) = (\lam(t))^{\frac12}v(t,\lam(t)x)$, we get
\begin{align*}
\int_{|x|<r}|u|^2 dx &= \int_{|y|<r\lam(t)}(\lam(t))^{-2}|v(t,y)|^2 dy\\
&= (\lam(t))^{-2}\int_{|y|<\varepsilon r \lam(t)}|v(t,y)|^2 dy + (\lam(t))^{-2}\int_{\varepsilon r \lam(t)<|y|<r\lam(t)}|v(t,y)|^2 dy\\
&=: A + B,
\end{align*}
where $\varepsilon$ will be determined later. Fixing $r$, by H\"older's inequality, we obtain
$$
A \le C\varepsilon^2r^2\|v(t)\|_{L_x^{p+1}}^2 \le C \varepsilon^2r^2\|Q_b\|_{\h}^2 < \frac\varepsilon2
$$
for any $\varepsilon < \frac12(C r^2\|Q_b\|_{\h}^2)^{-1}$. On the other hand, since $\lam(t) \to +\infty $ as $t\to T_+$, $B$ is estimated as
$$
B \le r^2\|v(t)\|_{L^{p+1}(|y|\ge\varepsilon r\lam(t))}^2 < \frac\varepsilon2 \;\;(t \;\;\mbox{close to}\;\; T_+)
$$
by Remark \ref{cpt}. Therefore we get \eqref{getL2}.

Since $|y_r'(t)|\le C$ from \eqref{y'-bound}, we have $y_r(0)\le y_r(t) + Ct$ for all $t \in [0, T_+)$ and thus
$$
y_r(0) \le \lim_{t\to T_+}y_r(t) + CT_+ = CT_+.
$$
Taking the limit $r \to \infty$, we get that $\vp\in L^2$. For any $\varepsilon > 0$, choose $\al$ small enough that
$$
\int_{T_+-\al}^{T_+}|y_r'|dx \le C\al < \frac{\varepsilon}{2}.
$$
Then the conservation of mass (Remark \ref{m-e-cons}) and \eqref{getL2} yield
\begin{align*}
\|\vp\|_{L_x^2}^2 = \|u(T_+-\al)\|_{L_x^2}^2 &\le \|u(T_+-\al)\|_{L_x^2(|x|<r)}^2 + \frac{\varepsilon}{2} \le y_r(T_+ -\al) + \frac{\varepsilon}{2}\\
&\le \lim_{t\to T_+}\int_{t}^{T_+-\al}y_r'(s)ds +\frac{\varepsilon}{2} < \varepsilon
\end{align*}
for a large $r$.
Since $\varepsilon$ is arbitrary, we have that $\vp=0$, which contradicts $T_+<\infty$.

\subsection{Case: $T_+ = \infty$}\label{case2}

Choose $\de > 0$ such that $\eg(\vp) \le (1-\de)\eg(Q_b)$. Then from Lemma \ref{trapping} (energy trapping) and compactness we deduce that there exists $r_0 > 0$ such that for $r > r_0$ and $t\in[0,\infty)$,
\begin{align}\label{lowbound-e}
\int_{|x| < r}(|\nabla u|^2 -g|u|^{p+1})\,dx \ge C_{\de}\|\vp\|_{\h}^2 > 0.
\end{align}

Let $b(x)\in C_0^\infty(\rt)$  and $b_r(x)$ be as follows:
\begin{align*}b(x)&:=\left\{ \begin{array}{cc}
|x|^2 & (|x|\leq 1) \\
0 & (|x|\geq 10)
\end{array}  \right.,\qquad
b_r(x):=r^2b(\frac{x}{r}).
\end{align*}
Set $z_r(t) = \int b_r|u(t)|^2 dx$. Then from the density by $H^2$ data and continuous dependency of solutions it follows that
\begin{align}\label{dilation}
\frac{d}{dt}z_r = 2 \textrm{Im}\int \nabla b_r \cdot \nabla u \bar{u}~dx
\end{align}
and
\begin{align}\begin{aligned}\label{lvirial}
\frac{d^2}{dt^2}z_r &= 2 \textrm{Im}\int \left[-\Delta b_r u_t \bar{u}-\left(\nabla b_r \cdot \nabla\bar{u}\right)u_t +\left(\nabla b_r \cdot \nabla u \right)\bar{u_t}\right] dx\\
&= 4\textrm{Re}\int(\nabla^2 b_r \cdot \nabla\bar{u})\nabla u dx  -\frac{2p-2}{p+1}  \int(\Delta b_r)g|u|^{p+1}dx\\
&\qquad + \frac{4}{p+1}\int\left(\nabla b_r \cdot \nabla g\right)|u|^{p+1} dx - \int (\Delta^2b_r)|u|^2 dx.
\end{aligned}\end{align}
\eqref{lvirial} has been obtained without regard to the radial symmetry.

By \eqref{dilation} we deduce that
$$
|z_r'(t)|\le Cr^2\int|\nabla u|^2dx \le C_{\de_0}r^2\|Q_b\|_{\h}^2.
$$
Since $x\cdot \nabla g \ge -bg$, by Remark \ref{cpt}, \eqref{lvirial} can be estimated as follows:
\begin{align}\begin{aligned}\label{z''}
z_r''(t) &\ge \underset{|x|\leq r}{\int}\left( 8|\nabla u|^2 -\frac{12p-12}{p+1}g|u|^{p+1} \right) +\frac{8}{p+1} \underset{|x|\leq r}{\int} (x\cdot\nabla g)|u|^4 dx\\
         &\qquad  \qquad \qquad  - C\left[\underset{r \le |x| \le 10r}{\int}|\nabla u|^2 +g|u|^{p+1} + \frac{|u|^2}{|x|^2}dx \right]\\
&\ge 8 \left[ \underset{|x|\leq r}{\int}|\nabla u|^2 -g|u|^{p+1}\right]  - C\left[\underset{r \le |x| \le 10r}{\int}|\nabla u|^2 +g|u|^{p+1} + \frac{|u|^2}{|x|^2}dx \right]\\
&\ge C_{\de_0}\|\vp\|_{\h}^2
\end{aligned}\end{align}
for $r$ sufficiently large. Therefore, by integrating \eqref{z''} over $[0, t]$, we get
\begin{align*}
|z_r'(t) - z_r'(0)| \le 2C_{\de_0}r^2\|Q_b\|_{\h}^2,\quad z_r'(t) - z_r'(0) \ge \left(\tilde{C}_{\de_0}\|\vp\|_{\h}^2\right)t,
\end{align*}
which are not compatible for arbitrarily large $t$. This completes the proof of Proposition \ref{rigidity} and hence Theorem \ref{mainresult} holds.

%%%%%%%%%%%%%%%%%%%%%%%%%%%%%%%%%%%%%%%%%%%%%%%%%%%%%%%%%%%%%%%%%%%%%%%%%%%%%%%%%%%%%%%%%%%%%%%%%%%%%%%%%%%%%%%%%%%%%%%%%%%%%%%%%%%%%%%%%%%%%%%%%%%%%%%%%%%%%%%%%%%%%%%%%%%%%%%%%%%%%%%%%%%%%%%%%%%%%%%%%%%%%%%%%%%%%%%%%%%%%%%%%%%%%%%%%%%%%%%%%%%%%%%%%%%%%%%%%%%%%%%%%%%%%%%%%%%%%%%%%%%%%%%%%%%%%%%%%%%%%%%%%%%%%%%%%%%%%%%%%%%%%%%%%%%%%%%%%%%%%%%%%%%%%%%%%%%%%%%%%%%%%%%%%%%%%%%%%%%%%%%%%%%%%%%%%%%%%%%%%%%%%%%%%%%%%%%%%%%%%%%%%%%%%%%%%%%%%%%%%%%%
\section{Blow-up: Proof of Theorem \ref{blowup-thm}}
In this section, we show the finite time blowup via localized virial identity \eqref{lvirial}. To this end, we introduce a variational estimate which is fundamental part of the proof for Theorem \ref{blowup-thm}.

\begin{lem}\label{vari-nega}
	Let $\eg(\varphi)<(1-\delta_0)\eg(Q_b)$ and $\gs\|\varphi\|_{\h}^2 \ge \|Q_b\|_{\h}^2$ for some $\delta_0 > 0$. Then there exists $\bar{\delta}=\bar{\delta}(\delta_0)$ such that
	\begin{align}\label{negative}
	&(i)  \;\gs\|\varphi\|_{\h}^2 \ge (1+ \bar{\delta})\|Q_b\|_{\h}^2,\\
	&(ii) \;\int (|\nabla \varphi|^2 - (1-\eta)g|\varphi|^{p+1})\, dx \le -\frac{(p_0-(p_0+1)\eta)\bd}{\gs}\|Q_b\|_{\h}^2,
	\end{align}
	for $0 \le \eta \le k_g$, where $k_g = \frac{p_0-g_0}{p_0+1-g_0}$, $g_0 = \gs(p_0+1 - g_i)$, and $p_0 = \frac{p-1}2$.
\end{lem}
\begin{proof}
	Let us invoke from the proof of Lemma \ref{vari-lem} that
	$$
	f(\bar{y}) \le(1-\de_0) f(y_0).
	$$
	where $f(y)= \frac12y - \frac{C_*^{p+1}}{p+1}y^{\frac{p+1}2}$, $\bar{y} = \gs \|\varphi\|_{\h}^2$, and $y_0 = \|Q_b\|_{\h}^2$.
	Since $y_0 \le \bar{y}$ and $f$ is strictly decreasing on $[y_0,+\infty)$, we find out that $ \bar{y} \ge (1+\bar{\delta})y_0$ for $\bar \delta$ as in Lemma \ref{vari-lem}.
	Hence we reach the following:
	\begin{align*}
	\int |\nabla\varphi|^2 -(1-\eta)g|\varphi|^{p+1}dx &= (p+1)(1-\eta)\eg(\vp) -(p_0- (p_0+1)\eta) \int |\nabla \vp|^2\\
	&\le ({p+1})(1-\eta)\eg(Q_b)-\frac{p_0- (p_0+1)\eta}{\gs}\|Q_b\|_{\h}^2-\frac{(p_0- (p_0+1)\eta)\bd}{\gs}\|Q_b\|_{\h}^2\\
	&\le \frac{g_0-p_0 -\eta(g_0-p_0-1)}{\gs}\|Q_b\|_{\h}^2 -\frac{(p_0- (p_0+1)\eta)\bd}{\gs}\|Q_b\|_{\h}^2\\
	&\le -\frac{((p_0+1-g_0))(k_g -\eta)}{\gs}\|Q_b\|_{\h}^2 -\frac{(p_0- (p_0+1)\eta)\bd}{\gs}\|Q_b\|_{\h}^2\\
    &\le -\frac{(p_0- (p_0+1)\eta)\bd}{\gs}\|Q_b\|_{\h}^2.
	\end{align*}
	Therefore we get the desired result.
\end{proof}
The following holds immediately from the Lemma \ref{vari-nega} together with the energy conservation and continuity argument.
\begin{cor}\label{blowup-cor}
	Let $u$ be a solution of \eqref{maineq} with $\varphi$  such that
	$$
	E_g(\varphi) \le (1-\delta_0) E_g(Q_b) \;\;\mbox{and}\;\; \gs \|\varphi\|_{\h}^2 \ge \|Q_b\|_{\h}^2
	$$
	for some $\delta_0 > 0$. Then there exits $\bar{\delta}=\bar{\delta}(\delta_0)$ such that
	\begin{align*}
	&(i)   \;\gs\|u(t)\|_{\h}^2 \ge (1+\bar{\delta})\|Q_b\|_{\h}^2,\\
	&(ii)  \;\int |\nabla u(t)|^2 -(1-\eta) g|u(t)|^{p+1} dx \le -\frac{(p_0-(p_0+1)\eta)\bd}{\gs}\|Q_b\|_{\h}^2
	\end{align*}
	for all $t \in I^*$ and $0 \le \eta \le k_g$, where $I^*$ is the maximal existence time interval and $k_g$ is the same as in Lemma \ref{vari-nega}.
\end{cor}

\begin{proof}[Proof of Theorem \ref{blowup-thm}]
	We first show the part $(1)$. Let us invoke the localized virial identities \eqref{dilation} and \eqref{lvirial}. By integrating and taking limit $r\to \infty$ on both sides of \eqref{dilation} and \eqref{lvirial}, Fatou's lemma yields
	\begin{align*}
	\int |x|^2|u(t)|^2\,dx &\le 8\int_0^t\int_0^s\int\left(|\nabla u(t')|^2 - (1-k_g)g|u(t')|^{p+1}\right) dx dt'ds\\
	& \qquad \quad + 2 t{\rm Im}\int (\nabla \varphi \cdot x)\varphi\,dx + \int |x|^2|\varphi|^2\,dx.
	\end{align*}
	Then from Corollary \ref{blowup-cor} it follows that
	$$
	\int |x|^2|u(t)|^2\,dx \le - C_g\bd t^2 + 2 t{\rm Im}\int (\nabla \varphi \cdot x)\varphi\,dx + \int |x|^2|\varphi|^2\,dx
	$$
	for some constant $C_g$. The last inequality gives us that the maximal interval is bounded.
	%%%%%%%%%%%%%%%%%%%%%%%%%%%%%%%%%%%%%%%%%%%%%%%%%%%%%%%%%%%%%%%%%%%%%%%%%%%%%%%%%%%%%%%%%%%%%%%%%%%%%%%%%%%%%%%%%%%%%%%%%%%%%%%%%%%%%%%%%%%%%%%%%%%%%%%%%%%%%%%%%%%%%
	%%%%%%%%%%%%%%%%%%%%%%%%%%%%%%%%%%%%%%%%%%%%%%%%%%%%%%%%%%%%%%%%%%%%%%%%%%%%%%%%%%%%%%%%%%%%%%%%%%%%%%%%%%%%%%%%%%%%%%%%%%%%%%%%%%%%%%%%%%%%%%%%%%%%%%%%%%%%%%%%%%%%%
	
	For the part $(2)$, we need another $b_r$. Let us introduce the function $\beta \in C^4([0, \infty))$ such that $\beta(s) = s$ for $0 \le s \le 1$, smooth for $1 < s < 10$, and $0$ for $s \ge 10$ and further that
	$0 \le \beta \le 1$ and $\beta'(s) \le 1$ for all $s \ge 0$. For the construction of such function see Appendix B of \cite{bole}.
	
	Now let $\beta_r(s) = r \beta(\frac sr)$ and $b_r(|x|) = \int_0^{|x|}\beta_r(s)\,ds$. Then, by \eqref{lvirial} and radial symmetry of $u$, we get
	\begin{align}\begin{aligned}\label{lvirial2}
	\frac{d^2}{dt^2}z_r &= 4\int \beta_r'(|x|)|\nabla u|^2\,dx - \frac{2p-2}{p+1}\int \nabla \cdot \left(\frac{x}{|x|}\beta_r(|x|)\right)g|u|^{p+1}\,dx\\
	&\qquad  + \frac4{p+1}\int \frac{\beta_r(|x|)}{|x|}(x\cdot\nabla g) |u|^{p+1}\,dx - \int \Delta \nabla\cdot\left(\frac{x}{|x|}\beta_r(|x|)\right)|u|^2\,dx.
	\end{aligned}\end{align}
	%%%%%%%%%%%%%%%%%%%%%%%%%%%%%%%%%%%%%%%%%%%%%%%%%%%%%%%%%%%%%%%%%%%%%%%%%%%%%%%%%%%%%%%%%%%%%%%%%%%%%%%%%%%%%%%%%%%%%%%%%%%%%%%%%%%%%%%%%%%%%%%%%%%%%%%%%%%%%%%%%%%%%
	%%%%%%%%%%%%%%%%%%%%%%%%%%%%%%%%%%%%%%%%%%%%%%%%%%%%%%%%%%%%%%%%%%%%%%%%%%%%%%%%%%%%%%%%%%%%%%%%%%%%%%%%%%%%%%%%%%%%%%%%%%%%%%%%%%%%%%%%%%%%%%%%%%%%%%%%%%%%%%%%%%%%%
	Since $\beta_r'(s) \le 1$ and $x\cdot \nabla g \le (p+1)(k_g - \rho)g$, we then have
	\begin{align*}
	\frac{d^2}{dt^2} z_r(t) &\le 4{\int}\left(|\nabla u(t)|^2 - (1-k_g+\rho)g|u(t)|^{p+1}\right) dx\\
	&\qquad\qquad + 4 \int\left[1 - \frac{{p-1}}{2(p+1)}\nabla \cdot \left(\frac{x}{|x|}\beta_r(|x|)\right)\right]g|u|^{p+1} - \int \Delta \nabla\cdot\left(\frac{x}{|x|}\beta_r(|x|)\right) |u|^2dx\\
	&\le 4{\int}\left(|\nabla u(t)|^2 - (1-k_g+\rho)g|u(t)|^{p+1}\right) dx + C\left[\underset{|x| \ge r}{\int} g|u|^{p+1} + \frac{|u|^2}{|x|^2}dx \right]\\
	&\le 4{\int}\left(|\nabla u(t)|^2 - (1-k_g+\rho)g|u(t)|^{p+1}\right) dx + Cg_s\||x|^{-\frac{2b}{3p-7}}u\|_{L_x^\infty(|x| \ge r)}^{\frac{3p-7}2}\|u\|_{L_x^6}^{\frac{15-3p}4}\|u\|_{L_x^2}^{\frac{3+p}4}\\
&\qquad\qquad\qquad\qquad\qquad\qquad\qquad\qquad\qquad\quad\;\; + Cr^{-2}\|u\|_{L_x^2}^2.
	\end{align*}
	
	To control the second term we use the decay estimate of radial function $f$ \cite{stra, chooz-ccm}:
	$$
	\||x|^{\frac12}f\|_{L_x^\infty} \le C_0\|f\|_{L_x^2}^\frac12\|\nabla f\|_{L_x^2}^\frac12.
	$$
	The mass conservation (Remark \ref{m-e-cons}) gives us that
	\begin{align*}
	\frac{d^2}{dt^2} z_r(t) \le 4(1+\varepsilon(r))\int\left(|\nabla u(t)|^2 - \frac{(1-k_g+\rho)}{1+\varepsilon(r)}g|u(t)|^{p+1}\right) dx + Cr^{-2}\|\varphi\|_{L_x^2}^2,
	\end{align*}
	where $\varepsilon(r) =\frac14 CC_0^{\frac{3p-7}2}C_1^{\frac{15-3p}4}g_s\|\varphi\|_{L_x^2}^{p-1}r^{-\frac{p+3}{6p-14}}$ where $C_1$ is Sobolev embedding constant $\h \hookrightarrow L_x^6$. Hence if we choose $r$ large enough, then since $\rho > 0$, by Corollary \ref{blowup-cor} we deduce that
	\begin{align*}
	\frac{d^2}{dt^2} z_r(t) \;\le\; -\frac{C_g\bd}2.
	\end{align*}
	By the same argument as of $(1)$ we obtain the desired result.

\end{proof}
%%%%%%%%%%%%%%%%%%%%%%%%%%%%%%%%%%%%%%%%%%%%%%%%%%%%%%%%%%%%%%%%%%%%%%%%%%%%%%%%%%%%%%%%%%%%%%%%%%%%%%%%%%%%%%%%%%%%%%%%%%%%%%%%%%%%%%%%%%%%%
%%%%%%%%%%%%%%%%%%%%%%%%%%%%%%%%%%%%%%%%%%%%%%%%%%%%%%%%%%%%%%%%%%%%%%%%%%%%%%%%%%%%%%%%%%%%%%%%%%%%%%%%%%%%%%%%%%%%%%%%%%%%%%%%%%%%%%%%%%%%%%
\section{Appendix}

In this section we prove the nonexistence of  positive radial solution to \eqref{g-ell} by the same argument of \cite{dini}. Let us consider the second order ODE:
\begin{align}\label{ode}
\left\{
\begin{array}{l}
Q_{rr} + \frac2rQ_r + g(r)Q^p = 0,\;\;\mbox{in}\;\;(0, \infty),\\
Q_r(0) = 0, \; Q(0) = Q_0  > 0,
\end{array}\right.
\end{align}
where $p = 5-2b$ and $0 < b < 2$.

\begin{prop}\label{noex}
Let $g$ satisfy the conditions \eqref{scaling} with $0 < b < 1$. Set $H(r) = \int_0^r s^{3-b}(s^b g)'\,ds$. Suppose that $H(r) \ge 0$ for all $r \ge 0$ and there exists $R > 0$ such that $H(r) \ge H(R) > 0$ for all $r \ge R$. Then every global solution $Q$ to \eqref{ode} must have a finite zero in $(0, \infty)$.
\end{prop}
\begin{rem}\label{b>1-rem}
The condition $b < 1$ is used to give a meaning to $Q_r(0)$. If $g \in C^1(\mathbb R^3)$, then Proposition \ref{noex} holds for $1 \le b < 2$.
\end{rem}

\begin{rem}\label{noex-rem}
If $g$ satisfies the conditions \eqref{scaling} and \eqref{rig-con} $((r^bg)' \ge 0)$, then the corresponding $H$ satisfies the assumption of the proposition. Therefore \eqref{g-ell} has no positive radial solutions in $\mathbb R^3$.
\end{rem}

\begin{proof}
We prove by contradiction. Suppose that there exists $Q_0 > 0$ such that the solution $Q$ remains positive in $(0, \infty)$.
Then from the equation \eqref{ode} it follows that
$$
Q_r(r) = -\int_0^r\left(\frac sr\right)^2 g(s) Q(s)^p\,ds \le 0.
$$
Thus $Q$ is monotonically decreasing.

Now set
$$
V(r) = \frac1{p+1}\int_0^r s^{3-b}(s^bg)' Q^{p+1}\,ds
$$
Let $R_0$ be the last zero of $H$ in $(0, \infty)$. Then $R_0 < R$. By the Second Mean Value Theorem we have that for $r \ge R$
\begin{align*}
V(r) &= \frac1{p+1}\left(\int_0^{R_0}  + \int_{R_0}^R + \int_R^r\right) s^{3-b}(s^bg)' Q^{p+1}\,ds\\
&= \frac1{p+1}(Q_0^{p+1}H(r_0) + Q^{p+1}(R_0)(H(r_1) - H(R_0)) + Q^{p+1}(R)(H(r_2) - H(R)))\\
&\ge \frac1{p+1}Q^{p+1}(R_0)H(r_1) =: \delta,
\end{align*}
where $r_0 \in [0, R_0]$, $r_1 \in (R_0, R] $, and $r_2 \in [R, r]$. Here $r_1$ should be strictly greater than $R_0$ because 
\begin{align*}
&Q^{p+1}(R_0)\int_{R_0}^{r_1}s^{3-b}(s^bg)' Q^{p+1}\,ds = \int_{R_0}^R s^{3-b}(s^bg)' Q^{p+1}\,ds\\
 &\qquad = H(R)Q(R) - H(R_0)Q(R_0) - (p+1)\int_{R_0}^R H(s)Q^p(s)Q'\,ds \ge H(R)Q(R) > 0.
\end{align*} 
Since $R_0$ and $R$ are fixed, $\delta > 0$ for all $r \ge R$.

To describe $V$ in detail let us introduce Pohozaev identity:
\begin{lem}[see Lemma 3.7 of \cite{dini}]\label{poh}
Let $\Omega$ be a bounded smooth domain in $\mathbb R^3$. Let $Q$ be classical solution to \eqref{g-ell}. Then we have
\begin{align*}
&\int_{\Omega}\left(\frac{b}{p+1} g Q^{p+1} + \frac1{p+1}(x\cdot \nabla g)Q^{p+1}\right)\,dx\\
&\qquad = \int_{\partial \Omega} \left((x\cdot \nabla Q)\frac{\partial Q}{\partial \nu} - (x\cdot \nu)\frac{|\nabla Q|^2}2 + \frac1{p+1}(x\cdot \nu)gQ^{p+1} + \frac12 Q\frac{\partial Q}{\partial\nu}\right)\,dS,
\end{align*}
where $dS$ is the volume element of $\partial \Omega$ and $\nu$ is the unit outer normal vector on $\partial \Omega$.
\end{lem}
If $\Omega$ is the ball with radius $r$ centered at the origin, then the Pohozaev identity shows that
\begin{align}\begin{aligned}\label{vr}
V(r) &= \frac1{p+1}\int_0^r s^3 g' Q^{p+1}\,ds + \frac{b}{p+1}\int_0^r s^2 g Q^{p+1}\,ds\\
&= \frac1{|S^2|} \int_{\Omega}\left(\frac{b}{p+1}g Q^{p+1} + \frac1{p+1}(x\cdot \nabla g)Q^{p+1}\right)\,dx\\
&= 3r^3Q_r^2(r) + r^3g(r)Q^{p+1}(r) + 2r^2Q(r)Q_r(r) \ge \delta > 0
\end{aligned}\end{align}
for all $r \ge R$.

Since $g$ satisfies \eqref{scaling}, by Theorem 3.35 of \cite{ni} we can deduce that $$W(r) := r^\frac12 Q(r)$$ is bounded for $r \ge R$.
We can also show that $W$ has no local minimum in $(R, \infty)$. In fact, if $W$ has a local minimum at $\rho > R$, then
$$
W'(\rho) = \frac12\rho^{-\frac12}Q(\rho) + \rho^\frac12 Q'(\rho) = 0
$$
and
$$
W''(\rho) = -\frac1{4}\rho^{-\frac32} Q(\rho) + \rho^{-\frac12}Q'(\rho) + \rho^\frac12 Q''(\rho) \ge 0.
$$
Combining these with \eqref{ode}, we have
\begin{align}\label{lowQ}
\rho^3 g(\rho)Q^{p+1}(\rho) \le \frac14\rho Q^2(\rho).
\end{align}
On the other hand, from \eqref{vr} it follows that
$$
\rho^3g(\rho)Q^{p+1}(\rho) \ge -3\rho^3(Q'(\rho))^2 - 2\rho^2Q(\rho)Q'(\rho) + \delta = \frac14\rho Q^2(\rho) + \delta.
$$
This contradicts \eqref{lowQ}.

Now we may assume that $W(r) \to A$ as $r \to \infty$ for some constant $A \ge 0$ because $W$ is bounded, positive, and has no local minimum.
From this we deduce that $\liminf_{r \to \infty}|W_1(r)| = 0$, where $W_1(r) = rW'(r)$. Lemma 5.25 of \cite{ni} gives us  that there exists a sequence $r_j \to \infty$ such that
$$
W_1(r_j) \to 0\;\;\mbox{and}\;\; r_jW_1'(r_j) \to 0\;\;\mbox{as}\;\;j\to \infty.
$$
Since
\begin{align*}
&W_1(r_j) = \frac12 W(r_j) + r_j^\frac32 Q'(r_j) \to 0,\\
&r_jW_1'(r_j) = W_1(r_j) + r_j^2W''(r_j) \to 0,
\end{align*}
we get
$$
r_j^2W''(r_j) \to 0, r_j^\frac32 Q'(r_j) \to -\frac12A, \;\;\mbox{and}\;\; r_j^\frac52Q''(r_j) \to \frac34 A.
$$
From \eqref{ode} it follows that
$$
r_j^{3}g(r_j)Q^{p+1}(r_j) \to \frac14 A^2.
$$
Invoking \eqref{vr}, we have
\begin{align*}
0 < \delta \le 3r_j^{3}Q'^2(r_j) + r_j^{3}g(r_j)Q^{p+1}(r_j) + 2r_j^{2}Q(r_j)Q'(r_j) \to 0,
\end{align*}
which reach a contradiction. This completes the proof of Proposition \ref{noex}.

\end{proof}

\section*{Acknowledgements}
This work was supported by NRF-2018R1D1A3B07047782(Republic of Korea).
%%%%%%%%%%%%%%%%%%%%%%%%%%%%%%%%%%%%%%%%%%%%%%%%%%%%%%%%%%%%%%%%%%%%%%%%%%%%%%%%%%%%%%%%%%%%%%%%%%%%%%%%%%%%%%%%%%%%%%%%%%%%%%%%%%%%%%%%%%%%%
%%%%%%%%%%%%%%%%%%%%%%%%%%%%%%%%%%%%%%%%%%%%%%%%%%%%%%%%%%%%%%%%%%%%%%%%%%%%%%%%%%%%%%%%%%%%%%%%%%%%%%%%%%%%%%%%%%%%%%%%%%%%%%%%%%%%%%%%%%%%%

%%%%%%%%%%%%%%%%%%%%%%%%%%%%%%%%%%%%%%%%%%%%%%%%%%%%%%%%%%%%%%%%%%%%%%%%%%%%%%%%%%%%%%%%%%%%%%%%%%%%%%%%%%%%%%%%%%%%%%%%%%%%%%%%%%%%%%%%%%%%%%%%%%%%%%%%%%%%%%%%%%%%%%%%%%%%%%%%%%%%%%%%%%%%%%%%%%%%%%%%%%%%%%%%%%%%%%%%%%%%%%%%%%%%%%%%%%%%%%%%%%%%%%%%%%%%%%%%%%%%%%%%%%%%%%%%%%%%%%%%%%%%%%%%%%%%%%%%%%%%%%%%%%%%

\end{document}